\documentclass[11pt]{article}
\usepackage{tikz}
\usetikzlibrary{positioning, arrows.meta, shapes.geometric}
\usepackage[a4paper, margin=2.3cm]{geometry}
\usepackage{amsmath,amssymb,amsthm, comment, enumitem, array}
\usepackage{color}
\usepackage{svg}
\usepackage{parskip}
\usepackage{hyperref}
\usepackage{setspace}
\usepackage{graphicx}
\usepackage{mathtools}
\usepackage[thinc]{esdiff}
\usepackage[euler]{textgreek}

\usepackage{babel}

\usepackage[capitalise, nameinlink]{cleveref}
\crefname{equation}{}{}
\crefname{figure}{{\sc Figure}}{{\sc Figure}}
\crefname{subsection}{Subsection}{Subsections}

\newtheorem{theorem}{Theorem}[section]
\newtheorem{example}{Example}[section]

\newtheorem{remark}[theorem]{Remark}

\newtheorem{lemma}[theorem]{Lemma}

\newtheorem{definition}[theorem]{Definition}

\newtheorem*{definition*}{Definition}

\newcommand{\F}{\mathbb{F}}

\newcommand{\PP}{\mathbb{P}}
\newcommand{\HH}{\mathbb{H}}
\newcommand{\cD}{\mathcal{D}}
\newcommand{\cL}{\mathcal{L}}

\newcommand{\Fq}{\mathbb{F}_q}

\newcommand{\bfp}{\mathbf{p}}
\newcommand{\bfv}{\mathbf{v}}

\begin{document}

\title{Horizontal Kakeya maximal operators in finite Heisenberg groups: Exact exponents and applications\thanks{
\hspace*{0.45cm} MSC: primary 42B25; Secondary: 43A75, 05B25, 51E20}}
\author{Thang Pham\thanks{Institute of Mathematics and Interdisciplinary Sciences, Xidian University. \newline
\hspace*{0.45cm} Email: {\tt thangpham.math@gmail.com}} \and Andrea Pinamonti \thanks{Department of Mathematics, University of Trento, via Sommarive 14, 38123 Povo, Italy. \newline
\hspace*{0.45cm} Email: {\tt andrea.pinamonti@unitn.it}}\and Dung The Tran \thanks{VNU University of Science, Hanoi, Vietnam. ~Email: {\tt tranthedung56@gmail.com}}\and Boqing Xue \thanks{Corresponding author, Institute of Mathematical Sciences, ShanghaiTech University. \newline
\hspace*{0.45cm} Email: {\tt xuebq@shanghaitech.edu.cn}}}
\date{}
\maketitle

{  \begin{abstract}
Let \(q\) be an odd prime power. We study Kakeya maximal operators associated with horizontal lines in the finite Heisenberg groups \(\mathbb H_n(\mathbb F_q)\). Our principal object is the refined-direction maximal operator, whose parameter records the projective horizontal direction together with the central homogeneous coordinate determined by horizontality. In rank one, we prove
\[
\|\mathcal M_{\mathbb H_1}^{\mathrm{rd}}F\|_{\ell^2(\mathcal D_1)}
\lesssim q^{\frac{1}{2}}
\|F\|_{\ell^2(\mathbb H_1(\mathbb F_q))},
\]
where the exponent \(\frac{1}{2}\) is sharp. Combining this estimate with endpoint bounds and interpolation, we determine the exact mixed-norm growth exponent:
\[
A^{\mathrm{rd}}_1(u,v)
=
\max\left\{
\frac1v,\,
1-\frac1u,\,
\frac2v-\frac1u,\,
1+\frac2v-\frac3u
\right\},
\qquad 1\le u,v\le\infty.
\]
As a consequence, if \(E\subset\mathbb H_1(\mathbb F_q)\) meets, in at least \(m\) points, a horizontal line in each refined direction from
\(\Omega\subset\mathcal D_1\), then
\[
|E|\gtrsim \frac{m^2|\Omega|}{q}.
\]

As a benchmark, we also analyze the coarser operator parameterized only by projective horizontal directions and determine its exact \(\ell^u\to\ell^v\) growth exponent in every rank. In rank one, this benchmark is established by a self-contained \(TT^*\) argument rather than polynomial vanishing, and the same planar estimate reappears as the zero-central-frequency component of the refined-direction proof. The nonzero central frequencies are controlled by Plancherel, character orthogonality, and a bounded-fiber property of an explicit quadratic map. Thus the sharp refined-direction estimate is obtained by purely Fourier-analytic methods.
\end{abstract}}
\section{Introduction}

The classical Kakeya needle problem asks for the smallest planar region in which a unit line
segment can be rotated through all directions. While the original formulation is geometric, a major modern development is the study of \emph{Kakeya (Besicovitch) sets}:
subsets of $\mathbb{R}^n$ that contain a unit line segment in every direction. Besicovitch \cite{Besicovitch1928}
showed that such sets can have Lebesgue measure zero, revealing that containing all
directions is compatible with extreme measure-theoretic thinness.

Despite this, Kakeya sets are widely expected to be large in a finer sense. The
\emph{Kakeya set conjecture} predicts that any Kakeya set $E\subset\mathbb{R}^n$ must have
Hausdorff and Minkowski dimension~$n$. This conjecture is a central organizing problem because
it quantifies the limits of how efficiently long thin objects in many directions can overlap,
and it is intertwined with major questions in harmonic analysis.

\medskip
\noindent\textbf{The Kakeya maximal operator.}
A convenient analytic formulation uses tubes at scale $\delta\in(0,1)$. Let $T$ denote the
$\delta$-neighbourhood of a unit line segment (so $T$ has length $1$ and cross-sectional
radius $\delta$). For $f\in L^1_{\mathrm{loc}}(\mathbb{R}^n)$ and $e\in\mathbb{S}^{n-1}$,
define the \emph{Kakeya maximal operator}
\begin{equation}\label{eq:euclid-kakeya-max}
K_\delta f(e)\ :=\ \sup_{T\parallel e}\ \frac{1}{|T|}\int_T |f(\mathbf{x})|\,d \mathbf{x},
\end{equation}
where the supremum runs over all such tubes $T$ with axis parallel to $e$.
The (scale-invariant) \emph{Kakeya maximal operator conjecture} asserts that
\begin{equation}\label{eq:euclid-kakeya-conj}
\|K_\delta f\|_{L^n(\mathbb{S}^{n-1})}\ \lesssim_{\varepsilon,n}\ \delta^{-\varepsilon}\,
\|f\|_{L^n(\mathbb{R}^n)}\qquad\text{for all }\varepsilon>0,
\end{equation}
and is closely related (indeed essentially equivalent, after discretization) to quantitative
overlap bounds for direction-separated tube families. Such bounds imply sharp lower estimates
for the Minkowski/Hausdorff size of Kakeya sets.

In dimension $n=2$, C\'ordoba \cite{Cordoba1977} proved sharp $L^2$-type estimates for $K_\delta$. In higher dimensions, Wolff's seminal work \cite{Wolff1995}
obtained improved $L^p$ bounds for Kakeya-type maximal operators, yielding strong lower bounds
for the possible dimension of Kakeya sets and setting the stage for much subsequent progress. Kakeya sets also arise as obstructions in Fourier analysis; for instance, Fefferman's counterexample to the ball multiplier problem is historically linked to Besicovitch-type configurations \cite{Fefferman1971}. Recent progress has clarified parts of the Euclidean picture, especially in three dimensions: Katz and Zahl \cite{KatzZahl2019} obtained an improved lower bound $\dim_H(E)\ge \frac52+\varepsilon_0$ for Besicovitch sets in $\mathbb{R}^3$, and Wang and Zahl \cite{WangZahl2025} proved the three-dimensional Kakeya conjecture.

\medskip
\noindent\textbf{The finite field model (abelian).}
A useful companion viewpoint replaces Euclidean space by a vector space 
over a finite field. Let $\mathbb{F}_q$ be the field with $q$ elements. 
A set $K \subset \mathbb{F}_q^n$ is a \emph{finite field Kakeya set} if it 
contains a full affine line in every direction: for every $[\mathbf{v}] \in 
\mathbb{P}^{n-1}(\mathbb{F}_q)$, there exists 
$\mathbf{y} \in \mathbb{F}_q^n$ such that
\[
\{\mathbf{y} + s\mathbf{v} : s \in \mathbb{F}_q\} \subset K.
\]
The finite field analogue of the Kakeya conjecture (see \cite{MT04} for example) asks whether such a set must have size comparable to the whole space, i.e.\ $|K| \ge c_n q^n$ with 
$c_n$ independent of $q$. In \cite{Dvir2009}, Dvir's breakthrough resolution of this problem introduced the polynomial method into Kakeya theory. 
Building on this method, Ellenberg, Oberlin, and Tao \cite{EOT} proved sharp Kakeya-type maximal estimates associated to algebraic varieties over finite 
fields, highlighting how algebraic structure may replace multiscale harmonic analysis.

While highly effective, the polynomial method relies fundamentally on algebraic vanishing. A natural complementary objective is to develop Fourier-analytic and geometric approaches that do not rely on algebraic vanishing, and therefore, extend to settings lacking strong algebraic structure.

In this paper, we pursue this program in the model setting of the finite Heisenberg group.

\medskip
\noindent\textbf{Kakeya sets and maximal operators in the Heisenberg group.}
In the sub-Riemannian geometry of the first Heisenberg group $\mathbb{H}_1$, a natural
analogue of the Kakeya problem asks how small a set can be if it contains horizontal unit line segments in every horizontal direction. Studying Kakeya phenomena in the Heisenberg group provides a natural test case for understanding which aspects of classical Kakeya behavior depend on Euclidean structure and which persist in a nonabelian setting with anisotropic dilations. The interaction between the group law, anisotropic dilations, and the horizontal distribution produces genuinely different geometric and analytic phenomena, making the Heisenberg group a natural setting for testing which Kakeya phenomena
persist beyond the Euclidean framework. Here the geometry is that determined by the canonical generating directions (the first layer of the stratification), and size is measured in terms of Hausdorff dimension with respect to a homogeneous metric such as the Kor\'{a}nyi metric.

In this direction, Liu \cite{Liu22} introduced a notion of Heisenberg Kakeya set and proved the sharp lower
bound $\dim_{\mathbb{H}}(E)\ge 3$ for Kakeya sets in $\mathbb{H}_1$, with sharpness obtained
(under this definition) by the horizontal plane. A complementary analytic
viewpoint is provided by Kakeya-type maximal operators: Venieri \cite{Venieri14} showed that, in analogy with
the Euclidean theory, suitable $L^p$ bounds for a Kakeya maximal operator imply lower bounds
for the Heisenberg Hausdorff dimension of (bounded) Besicovitch/Kakeya sets.
More recently, F\"assler, Pinamonti, and Wald \cite{FPW25} defined Heisenberg Kakeya maximal operators $M_\delta$
by averaging over $\delta$-neighbourhoods of horizontal unit segments (Kor\'{a}nyi tubes) and
proved an essentially optimal estimate in $\mathbb{H}_1$, recovering Liu's sharp dimension
bound as an application. Related techniques appear in \cite{FP22,HS,Zhang24}.

These developments motivate a finite field Heisenberg Kakeya problem adapted to horizontal geometry. Our aim is to develop a discrete counterpart of Heisenberg maximal-operator theory.

\medskip
\noindent\textbf{The finite field Heisenberg maximal operator.}
Let $q$ be an odd prime power. We identify the Heisenberg group $\mathbb{H}_n(\mathbb{F}_q)$, of rank $n$ and over $\mathbb{F}_q$,  with
\[
\mathbb{H}_n(\mathbb{F}_q):=\mathbb{F}_q^n\times \mathbb{F}_q^n\times \mathbb{F}_q
\]
equipped with the group law
\[
(\mathbf{x}, \mathbf{y}, t)\cdot(\mathbf{x}', \mathbf{y}', t')
=
\Bigl(\mathbf{x}+\mathbf{x}',\,\mathbf{y}+ \mathbf{y}',\,
t+t'+\bigl(\mathbf{x}\cdot \mathbf{y}'-\mathbf{y}\cdot \mathbf{x}'\bigr)\Bigr),
\]
where $\mathbf{x}\cdot \mathbf{y}'$ denotes the standard dot product on $\mathbb{F}_q^n$.
A \emph{horizontal direction} is an element $[\mathbf{v}]\in \mathbb{P}^{2n-1}(\mathbb{F}_q)$,
represented by a nonzero vector $\mathbf{v}=(\mathbf{a},\mathbf{b})\in
\mathbb{F}_q^n\times\mathbb{F}_q^n$ up to scaling. For
$\mathbf{p}=(\mathbf{x}_0,\mathbf{y}_0,t_0)\in \mathbb{H}_n(\mathbb{F}_q)$ and
$\mathbf{v}=(\mathbf{a},\mathbf{b})\ne \mathbf{0}$, we define the horizontal line through $\mathbf{p}$ in direction $[\mathbf{v}]$ as the right coset
\[
L_{\mathbf{p}, [\mathbf{v}]} := \{\, \mathbf{p}\cdot (s\mathbf{a}, s\mathbf{b}, 0) : s \in \mathbb{F}_q \,\}.
\footnote{When the basepoint $\mathbf{p}$ is not essential, we may write $L_{[\mathbf{v}]}$ by an abuse of notation. In certain proofs, we also adopt the notation $L_{\omega, \tau}$ for horizontal lines where their defining formula is clear from context.}\]
A direct computation gives the parametrization
\begin{equation}\label{eq:line-param-main}
L_{\mathbf{p}, [\mathbf{v}]}
=
\Bigl\{\bigl(\mathbf{x}_0+s\mathbf{a},\ \mathbf{y}_0+s\mathbf{b},\
t_0+ s(\mathbf{x}_0\cdot\mathbf{b}-\mathbf{y}_0\cdot\mathbf{a})\bigr):\, s \in \mathbb{F}_q\Bigr\},
\end{equation}
and hence $|L_{\mathbf{p}, [\mathbf{v}]}|=q$. 
The term $s(\mathbf{x}_0\cdot\mathbf{b}-\mathbf{y}_0\cdot\mathbf{a})$ is the \emph{Heisenberg twist}: it depends on the basepoint $(\mathbf{x}_0,\mathbf{y}_0)$
and couples it to the direction $[\mathbf{v}]$ via the underlying symplectic form.
This twist is what distinguishes horizontal Kakeya from the standard affine Kakeya
problem on $\mathbb{F}_q^{2n+1}$.

We now define two maximal operators associated with horizontal lines.
They differ in how much directional information they retain, and the gap between them drives the main results of this paper.

Throughout this paper, our maximal operators are defined using sums along lines
(not averages). The exponent formulas below reflect this definition.

To avoid notational conflict, we consistently use $[\mathbf{v}]$ for projective directions and the letters
$u,v$ for exponents.

\smallskip

For $F:\mathbb{H}_n(\mathbb{F}_q)\to\mathbb{C}$, we define the associated maximal operator
\[
\mathcal{M}_{\mathbb{H}_n}F([\mathbf{v}])
:=
\max_{\mathbf{p}\in\mathbb{H}_n(\mathbb{F}_q)}\ \sum_{\mathbf x\in L_{\mathbf{p}, [\mathbf{v}]}} |F(\mathbf x)|,
\qquad
[\mathbf{v}]\in \mathbb{P}^{2n-1}(\mathbb{F}_q).
\]

This operator remembers only the spatial direction $[\mathbf{v}]=[\mathbf{a}:\mathbf{b}]$; in particular, it forgets the third projective coordinate of the whole direction. To capture the full Heisenberg geometry, we introduce a finer notion.

{ 
\begin{definition}\label{def:refined-dir-set}
Define the refined direction set
\[
\mathcal D_n
:=
\bigl\{[\mathbf a:\mathbf b:c]\in\mathbb P^{2n}(\mathbb F_q):\ (\mathbf a,\mathbf b)\neq(\mathbf 0,\mathbf 0)\bigr\}.
\]
\end{definition}

\begin{definition}\label{def:refined-direction}
Let \(L\) be a horizontal line in \(\mathbb H_n(\mathbb F_q)\). Choose a point
\[
\mathbf p_0=(\mathbf x_0,\mathbf y_0,t_0)\in L
\]
and choose a representative
\[
(\mathbf a,\mathbf b)\neq(\mathbf 0,\mathbf 0)
\]
of the spatial direction of \(L\). Define the refined direction of \(L\) by
\[
\mathrm{Dir}(L)
:=
[\mathbf a:\mathbf b:\mathbf x_0\cdot\mathbf b-\mathbf y_0\cdot\mathbf a]
\in \mathcal D_n.
\]
\end{definition}
This definition is independent of the choices. Indeed, replacing \(\mathbf p_0\) by another point of \(L\) changes
\((\mathbf x_0,\mathbf y_0)\) to \((\mathbf x_0+s\mathbf a,\mathbf y_0+s\mathbf b)\), and hence
\[
(\mathbf x_0+s\mathbf a)\cdot\mathbf b-(\mathbf y_0+s\mathbf b)\cdot\mathbf a
=
\mathbf x_0\cdot\mathbf b-\mathbf y_0\cdot\mathbf a.
\]
On the other hand, replacing the representative \((\mathbf a,\mathbf b)\) by \((\lambda\mathbf a,\lambda\mathbf b)\), where \(\lambda\in\mathbb F_q^\times\), multiplies all three homogeneous coordinates by \(\lambda\). Therefore the projective class \(\mathrm{Dir}(L)\) is well-defined.
}

Note that \(\mathcal D_n=\mathbb P^{2n}(\mathbb F_q)\setminus\{[0:\cdots:0:1]\}\), i.e.\ the projective space with the single vertical direction removed.
Moreover, for every \([\mathbf a:\mathbf b:c]\in\mathcal D_n\), there exists a horizontal line \(L\subset \mathbb H_n(\mathbb F_q)\) with
\(\mathrm{Dir}(L)=[\mathbf a:\mathbf b:c]\). The set $\mathcal{D}_n$ fibers over $\mathbb{P}^{2n-1}(\mathbb{F}_q)$ via the
projection $[\mathbf{a}:\mathbf{b}:c]\mapsto[\mathbf{a}:\mathbf{b}]$,
with each fiber having~$q$ elements.

\begin{definition}\label{def:slope-refined-kakeya}
A set \(E\subset \mathbb H_n(\mathbb F_q)\) is called a full-direction horizontal Heisenberg Kakeya set if for every
\(\omega\in\mathcal D_n\) there exists a horizontal line \(L\subset E\) with \(\mathrm{Dir}(L)=\omega\).
\end{definition}

Although \(\mathbb H_n(\mathbb F_q)\) and \(\mathbb F_q^{2n+1}\) coincide as sets, the preceding notion is not the same as the usual affine Kakeya
property in \(\mathbb F_q^{2n+1}\): in the Heisenberg setting one prescribes refined directions of horizontal lines, whereas affine Kakeya prescribes
ambient directions of arbitrary affine lines. See Examples~\ref{ex:affine-not-refined} and~\ref{ex:refined-not-affine}.

For \(F:\mathbb H_n(\mathbb F_q)\to\mathbb C\), we define the full refined-direction horizontal maximal operator
\begin{equation}\label{eq:Mrd}
(\mathcal{M}^{\mathrm{rd}}_{\mathbb H_n}F)(\omega)
:=
\max_{\substack{L\ \mathrm{horizontal}\\ \mathrm{Dir}(L)=\omega}}
\ \sum_{\mathbf p\in L}|F(\mathbf p)|,
\qquad
\omega\in\mathcal D_n.
\end{equation}
{ 
In particular, the refined-direction operator discriminates among horizontal lines that share the same spatial direction but have different last homogeneous refined coordinates. Indeed, 
\[
\mathcal{M}_{\mathbb H_n}F([\mathbf a:\mathbf b])
=
\max_{[\mathbf a:\mathbf b:c]\in\mathcal D_n}
\mathcal{M}^{\mathrm{rd}}_{\mathbb H_n}F([\mathbf a:\mathbf b:c]).
\]
}

We now make the following observation. The projection $\pi(\mathbf{x},\mathbf{y},t)=(\mathbf{x},\mathbf{y})$
maps each horizontal line $L_{\mathbf{p},[\mathbf{v}]}$ bijectively onto an affine line in $\mathbb{F}_q^{2n}$ with the same projective direction~$[\mathbf{v}]$.
By aggregating $F$ along the $t$-fibers one obtains a pointwise domination
$\mathcal{M}_{\mathbb{H}_n}F\leq M_{2n}G_{F,u}$,
where $M_{2n}$ is the affine (abelian) Kakeya maximal operator on $\mathbb{F}_q^{2n}$ and $G_{F,u}(\mathbf{x},\mathbf{y}):=
\|F(\mathbf{x},\mathbf{y},\cdot)\|_{\ell^u(\mathbb{F}_q)}$.\footnote{For a finite set $X$ and $1\leq p\leq \infty$, we use $\ell^p(X)$ to denote the normed vector space consisting of complex-valued functions on $X$, equipped with the standard $\ell^p$-norm based on the counting measure.} Direct computations reveal that
sharp upper bounds for $\mathcal{M}_{\mathbb{H}_n}$ follow from existing abelian estimates. Note that this reduction is not reversible (see Remark~\ref{goodremark}), since the projection forgets the $t$-twist. In particular, there are functions $F$ for which the projected abelian maximal operator is of size $q$ in every direction, whereas $\mathcal M_{\mathbb H_1}F$ remains uniformly $O(1)$. We therefore view the sharp exponent formulas for $\mathcal{M}_{\mathbb{H}_n}$
as a reference estimate: they quantify what can be proved when one insists on the projective direction parameter alone, and they motivate the
refined-direction framework.

For $1\le u,v\le \infty$, let $A_n(u,v)$ be the smallest non-negative real number, independent of the field size, such that there
exists $C_{n,u,v}>0$ with
\begin{equation}\label{eq:An_def}
\|\mathcal{M}_{\mathbb{H}_n}F\|_{\ell^v(\mathbb{P}^{2n-1}(\mathbb{F}_q))}
\le
C_{n,u,v}\, q^{A_n(u,v)}\,
\|F\|_{\ell^u(\mathbb{H}_n(\mathbb{F}_q))}
\end{equation}
for all $F:\mathbb{H}_n(\mathbb{F}_q)\to\mathbb{C}$.

\medskip

{ 
As a benchmark for the refined-direction theory, we first determine the exact value of \(A_1(u,v)\) for all \(1\leq u,v\leq\infty\).}

\begin{theorem} \label{thm:exact_A}
For all $1\le u,v\le\infty$, one has
\begin{equation}\label{eq:A-max-formula}
A_1(u,v)
=
\max\Bigl\{\frac1v,\ 1-\frac1u,\ 1+\frac1v-\frac{2}{u}\Bigr\}.
\end{equation}
\end{theorem}

{  The upper bounds in Theorem~\ref{thm:exact_A} may also be
deduced from the finite field Kakeya maximal theorem of
Ellenberg, Oberlin, and Tao~\cite{EOT}, together with the
elementary endpoint estimates, interpolation, and the
$\ell^p$ embeddings used below. We nevertheless include a
self-contained proof. Its principal input is the
$TT^*$ estimate
\[
\|\mathcal M_2 G\|_{\ell^2(\mathbb P^1(\mathbb F_q))}
\le \sqrt{2q}\,\|G\|_{\ell^2(\mathbb F_q^2)}.
\]

The reason for retaining this argument is methodological.
The same planar estimate reappears exactly in the
zero-central-frequency contribution to the refined-direction
operator in the proof of Theorem~\ref{thm:sr-l2-bound}.
Consequently, the rank-one projective-direction and
refined-direction estimates can both be developed without
polynomial vanishing. In higher rank, by contrast, we invoke
the theorem of Ellenberg, Oberlin, and Tao for the corresponding
diagonal affine Kakeya estimate on $\mathbb F_q^{2n}$.}

\begin{theorem} \label{thm:exact_Hn}
For all $1\le u,v\le \infty$, one has
\begin{equation}\label{eq:An}
A_n(u,v)
=
\max\Bigl\{\frac{2n-1}{v},\ 1-\frac1u,\ 1+\frac{2n-1}{v}-\frac{2n}{u}\Bigr\}.
\end{equation}
\end{theorem}

Theorems~\ref{thm:exact_A} and~\ref{thm:exact_Hn} capture what can be extracted from the projective direction parameter $[\mathbf{v}]$ alone. To access the genuinely Heisenberg content of the problem, we now turn to the refined-direction operator.

Since \(\mathbb H_1(\mathbb F_q)\) may be identified with \(\mathbb F_q^3\), the results of Ellenberg, Oberlin, and Tao
\cite[Theorem~1.3 and Remark~1.4]{EOT} imply the bounds
\[
\|\mathcal{M}^{\mathrm{rd}}_{\mathbb H_1}F\|_{\ell^3(\mathcal D_1)}
\lesssim q^{\frac{2}{3}}\,\|F\|_{\ell^3(\mathbb H_1(\mathbb F_q))},
\qquad
\|\mathcal{M}^{\mathrm{rd}}_{\mathbb H_1}F\|_{\ell^2(\mathcal D_1)}
\lesssim q\,\|F\|_{\ell^2(\mathbb H_1(\mathbb F_q))}.
\]
The \(\ell^3\to\ell^3\) estimate is optimal (see Example~\ref{ex:rd-l3-sharp}). By contrast, the \(\ell^2\to\ell^2\) bound is not sharp. In the main theorem, we obtain the optimal exponent. The proof is purely Fourier-analytic and does not use polynomial vanishing.

\begin{theorem}\label{thm:sr-l2-bound}
There exists an absolute constant \(C\) such that for 
all \(F:\mathbb H_1(\mathbb F_q)\to\mathbb C\),
\[
\|\mathcal{M}^{\mathrm{rd}}_{\mathbb H_1}F\|_{\ell^2(\mathcal D_1)}
\le
C\,q^{\frac{1}{2}}\,\|F\|_{\ell^2(\mathbb H_1(\mathbb F_q))}.
\]
Moreover, the exponent \(\frac{1}{2}\) is sharp.
\end{theorem}

The proof begins by linearizing the maximal operator and decomposing in the central Fourier variable.
The zero-frequency term reduces, after projection to $\mathbb{F}_q^2$, to the planar Kakeya maximal estimate.
The nonzero frequencies are controlled by Plancherel and character orthogonality, which reduce the key $\ell^2$ bound
to a counting problem for an explicit quadratic map.

For $1\le u,v\le \infty$, let $A^{\mathrm{rd}}_1(u,v)$ be the smallest non-negative real number, independent of the field size, such that there exists
$C_{u,v}>0$ with
\begin{equation}\label{eq:Ard-def-proof}
\|\mathcal{M}^{\mathrm{rd}}_{\mathbb H_1}F\|_{\ell^v(\mathcal D_1)}
\le
C_{u,v}\,q^{A^{\mathrm{rd}}_1(u,v)}\,
\|F\|_{\ell^u(\mathbb H_1(\mathbb F_q))}
\qquad\text{for all }F:\mathbb H_1(\mathbb F_q)\to\mathbb C.
\end{equation}

Combining the sharp $\ell^2$ estimate with endpoint bounds and interpolation, we determine the exact mixed-norm exponent formula.

\begin{theorem}\label{thm:exact-rd}
For all $1\le u,v\le\infty$, one has
\begin{equation}\label{eq:Ard-formula}
A^{\mathrm{rd}}_1(u,v)
=
\max\Bigg\{\frac1v,\ 1-\frac1u,\ \frac{2}{v}-\frac1u,\ 1+\frac{2}{v}-\frac{3}{u}\Bigg\}.
\end{equation}
\end{theorem}

The upper bounds in Theorem~\ref{thm:exact-rd} are obtained by interpolating the sharp $\ell^2$ estimate of Theorem~\ref{thm:sr-l2-bound} with endpoint bounds ($\ell^1\to\ell^1$, $\ell^1\to\ell^\infty$, $\ell^\infty\to\ell^\infty$) and $\ell^p$ embeddings on~$\mathcal{D}_1$. For sharpness, each of the four terms in~\eqref{eq:Ard-formula} is forced by an explicit test function: a point mass, the indicator of a single horizontal line, the indicator of two non-parallel affine lines lifted to $t=0$, and the constant function $F\equiv 1$, respectively.

To extend Theorem~\ref{thm:exact-rd} to higher rank $n\ge 2$, a natural first step is to establish the analogue of Theorem~\ref{thm:sr-l2-bound}, namely, an estimate of the form
\[
\ell^2\bigl(\mathbb H_n(\F_q)\bigr)\longrightarrow \ell^{2n}(\mathcal D_n).
\]
A straightforward adaptation of the $n=1$ Fourier/Plancherel strategy yields a bound on the scale of $q^{\frac n2}$, rather than the smaller scale $q^{\frac{2n-1}{2n}}$ suggested by the point-mass obstruction. Closing this gap therefore appears to require genuinely new input. {  We leave the case $n\ge 2$ for future work.}

{  One motivation for the refined-direction framework is to
separate the part of the affine Kakeya problem that is already
controlled by our Fourier-analytic estimate from the genuinely
non-horizontal obstruction. Let
$E\subset\mathbb F_q^3$ be an affine Kakeya set, identified
with a subset of $\mathbb H_1(\mathbb F_q)$, and decompose
\[
\mathcal D_1=\Omega_1\sqcup\Omega_2,
\]
where $\Omega_1$ consists of the refined directions for which
$E$ contains a horizontal line and
$\Omega_2=\mathcal D_1\setminus\Omega_1$.

By Theorem~\ref{thm:Nm-refined},
\[
|E|\gtrsim q\,|\Omega_1|.
\]
Hence $|\Omega_1|\gtrsim q^2$ implies $|E|\gtrsim q^3$.
As a cardinality statement, this conclusion is also a
consequence of the classical finite field Kakeya maximal
theorem; in the full-direction case one may simply adjoin a
single vertical line and apply the affine Kakeya theorem.
The point here is instead that the $\Omega_1$ contribution is
controlled by the Fourier-analytic refined-direction estimate
of Theorem~\ref{thm:sr-l2-bound}, without using polynomial
vanishing. Thus, a fully Fourier-analytic approach to the
affine Kakeya problem may concentrate on the complementary
case $|\Omega_2|\gtrsim q^2$.}

We now discuss some applications, which are of independent interest. In the first consequence, we obtain quantitative lower
bounds for horizontal Heisenberg Kakeya sets, both for the full and restricted families of directions.

\begin{theorem}
\label{thm:Nm-refined}
{ Let $\varnothing \neq \Omega\subset \mathcal D_1$,} and let $E\subset \mathbb H_1(\mathbb F_q)$.
Let $m$ be an integer with $1\le m\le q$.
Assume that for every $\omega\in\Omega$, there exists a horizontal line
$L\subset \mathbb H_1(\mathbb F_q)$ such that
\[
\mathrm{Dir}(L)=\omega,
\qquad\text{and}\qquad
|E\cap L|\ge m.
\]
For $1\le u< \infty$ and $1\le v\le \infty$, we have
\begin{equation}\label{eq:Nm-uv-lower}
|E|
\ \gtrsim_{u, v} m^{u}\,|\Omega|^{\frac{u}{v}}\, q^{-u\,A^{\mathrm{rd}}_1(u,v)}.
\end{equation}

In particular, for $(u,v)=(2,2)$, one has 
\begin{equation}\label{eq:Nm-22}
|E|\ \gtrsim  \frac{m^2\,|\Omega|}{q}.
\end{equation}
Hence, any full-direction horizontal Heisenberg Kakeya set in $\mathbb{H}_1(\mathbb{F}_q)$ is of size at least $\gtrsim q^3$.
\end{theorem}

{  The strength of Theorem~\ref{thm:Nm-refined} is its
restricted-direction and partial-incidence formulation:
$\Omega$ may be an arbitrary subset of refined directions,
and the selected line in each direction is required to contain
only $m$ points of $E$. Thus, the theorem gives a
finite field Furstenberg-type estimate rather than only a
full-direction Kakeya bound. In particular, the sharp
$\ell^2$ estimate of Theorem~\ref{thm:sr-l2-bound} yields
\[
|E|\gtrsim \frac{m^2|\Omega|}{q}.
\]
For $m=q$ and $\Omega=\mathcal D_1$, this recovers
$|E|\gtrsim q^3$ by a purely Fourier-analytic argument.
The latter full-direction cardinality bound is also implied
by the classical affine Kakeya theorem after adjoining a
vertical line; the additional content here is the uniform
estimate for arbitrary $\Omega$ and $m$.}

Let $E\subset \mathbb H_1(\mathbb F_q)$, and define
\[
M_E(\omega):=
(\mathcal M^{\mathrm{rd}}_{\mathbb H_1}\mathbf 1_E)(\omega)
=
\max_{\substack{L\ \mathrm{horizontal}\\ \mathrm{Dir}(L)=\omega}}
|E\cap L|,
\qquad \omega\in\mathcal D_1.
\]
The next theorem provides higher moment bounds for refined directions.
\begin{theorem}\label{co-1mar}
For every $2\le s<\infty$, we have
\[
\sum_{\omega\in\mathcal D_1} M_E(\omega)^s
\lesssim_{s} q\,|E|^{s-1}.
\]
\end{theorem}
We conclude the introduction with brief remarks on related work. In the affine setting, Theorem~\ref{thm:Nm-refined} is connected to the study of Furstenberg sets in \cite{F1, F2} and the references therein. Fourier-analytic properties of affine Kakeya sets have been investigated recently by Fraser \cite{Fraser25}, and sum-product type results in the Heisenberg setting can be found in \cite{HH12, HH13, HH18, HH19}. We also note that Kakeya problems have been studied over rings such as $(\mathbb{Z}/N\mathbb{Z})^n$, where directions are defined modulo units and multi-scale phenomena arise, see \cite{Ars21, Dha24, DD21, DH13, HW18}.

\noindent\textbf{Organization of the paper.}
Section~\ref{section2} collects the basic notation and geometric facts, and proves the projection and domination principles that reduce horizontal Kakeya estimates on $\mathbb H_1(\mathbb F_q)$ to planar Kakeya estimates on $\mathbb F_q^2$.
Sections~\ref{section3}--\ref{section6} establish the sharp mixed-norm bounds for the horizontal Kakeya maximal operator $\mathcal{M}_{\mathbb H_1}$ in the case $n=1$ (Theorem~\ref{thm:exact_A}), including the matching lower bounds.
Section~\ref{section7} outlines the corresponding argument in higher rank, yielding the extension to general $n$ (Theorem~\ref{thm:exact_Hn}).
Sections~\ref{section8}--\ref{section10} are devoted to the refined-direction maximal operator on $\mathbb H_1(\mathbb F_q)$: we prove the sharp $\ell^2$ estimate, determine the full mixed-norm phase diagram, and deduce the stated combinatorial consequences for Heisenberg Kakeya sets.
Finally, Section~\ref{section11} presents examples distinguishing the refined-direction horizontal Kakeya property from the classical affine Kakeya property, and proposes an approach toward a new Fourier-analytic proof of the affine Kakeya theorem in $\mathbb{F}_q^3$.

\section{Preliminaries}\label{section2}

\subsection{Notation and Basic facts for $\mathbb H_1(\Fq)$}

Let $q$ be an odd prime power and let $\Fq$ denote the finite field with $q$ elements. The group $\mathbb{H}_1(\mathbb{F}_q)$ is identified with $\Fq^3$ with the group law
\[
(x, y, t)\cdot(x', y', t')
=
\bigl(x+x',\,y+ y',\,
t+t'+(xy'-yx')\bigr).
\]
For any $\bfp=(x_0,y_0,t_0)\in \HH_1(\Fq)$ and $[\bfv]\in \PP^1(\Fq)$ with $\bfv=(a,b)\neq (0,0)$, the horizontal line  
\[
L_{\mathbf{p}, [\mathbf{v}]}=\{\,\mathbf{p} \cdot(sa,sb,0):\ s\in \mathbb{F}_q\,\}=
\bigl\{(x_0+sa,\ y_0+sb,\
t_0+ s(x_0 b-y_0a)):\ s\in \mathbb{F}_q\bigr\}
\]
has refined direction $\mathrm{Dir}(L)=[a:b:(x_0b-y_0a)]\in \cD_1$, where $\cD_1=\PP^2(\Fq)\setminus\{[0:0:1]\}$. Given $(a,b)\neq (0,0)$, the map $(x,y)\mapsto (xb-ya)$ is a nontrivial linear functional and hence is surjective onto $\Fq$. 

Let us denote by $\cL$ the set of all horizontal lines in $\HH_1(\Fq)$. For $[\bfv] \in \PP^1(\Fq)$, $\omega\in \cD_1$, or $\bfp\in \HH_1(\Fq)$, denote by $\cL([\bfv])$, $\cL(\omega)$ and $\cL(\bfp)$ the set of all lines in $\cL$ with direction $[\bfv]$, with refined direction $\omega$, or through point $\bfp$, respectively. Note that
\[
|\HH_1(\Fq)|=q^3,\quad |\PP^1(\Fq)|=q+1,\quad |\cD_1|=q^2+q,\quad |L_{\bfp,[\bfv]}|=q.
\]
It follows that 
\[
|\cL|=\frac{q^3\cdot (q+1)}{q}=q^2(q+1).
\]
Given $[\bfv]\in \mathbb P^1(\Fq)$, one has \[
|\cL([\bfv])|=\frac{q^2(q+1)}{q+1}=q^2. 
\]
Given $\omega\in \cD_1$, one has 
\[
|\cL(\omega)|=\frac{q^2(q+1)}{q^2+q}=q.
\]
Given $\bfp\in \HH_1(\Fq)$, one has
\[
|\cL(\bfp)| = |\mathbb{P}^1(\mathbb{F}_q)|=q+1.
\]
Moreover, there is a natural one-to-one correspondence between $\cL(\bfp)$ and $\PP^1(\Fq)$. So each line in $\cL(\bfp)$ has different directions and different refined directions. 

\subsection{Basic lemmas in analysis}

The first lemma collects Propositions 6.11 and 6.12 of \cite{Fol}.

\begin{lemma} \label{lem:embed}
Let $X$ be a finite set with $|X|=N$ and let $1\le r\le s\le\infty$. Then for every $g:X\to\mathbb{C}$, one has 
\[
\|g\|_{\ell^s(X)}\le \|g\|_{\ell^r(X)},
\]
and 
\[
\|g\|_{\ell^r(X)} \le N^{\frac1r-\frac1s}\,\|g\|_{\ell^s(X)}.
\]
\end{lemma}

The second lemma is an interpolation, whose proof can be found in \cite{Grafakos04}.

\begin{lemma}\label{interpolation-inequality}
Let $X$ and $Y$ be finite sets, and let
\[
T : \ell^{p_0}(X) \cap \ell^{p_1}(X)\longrightarrow \ell^{r_0}(Y) \cap \ell^{r_1}(Y)
\]
be a linear operator.
Let $1 \le p_0, p_1, r_0, r_1 \le \infty$.
Assume that the following estimates hold for all functions $f : X \to \mathbb{C}$:
\[
\|Tf\|_{\ell^{r_0}(Y)} \le C_0 \, \|f\|_{\ell^{p_0}(X)},
\qquad
\|Tf\|_{\ell^{r_1}(Y)} \le C_1 \, \|f\|_{\ell^{p_1}(X)}.
\]
Then, for any $0 \le \theta \le 1$, we have
\[
\|Tf\|_{\ell^{r}(Y)} \le C_0^{1-\theta} C_1^{\theta} \, \|f\|_{\ell^{p}(X)},
\]
where the exponents $p$ and $r$ are defined by
\[
\frac{1}{p} = \frac{1-\theta}{p_0} + \frac{\theta}{p_1},
\qquad
\frac{1}{r} = \frac{1-\theta}{r_0} + \frac{\theta}{r_1}.
\]
\end{lemma}

\subsection{Planar domination lemma}

For $f:\mathbb{F}_q^2\to\mathbb{C}$, define
\begin{equation}\label{eq:Mpl}
\mathcal{M}_{2}f([\mathbf{v}])
:=
\max_{\ell\parallel [\mathbf{v}]}\ \sum_{\mathbf x\in\ell} |f(\mathbf x)|,
\qquad [\mathbf{v}]\in \mathbb{P}^{1}(\mathbb{F}_q),
\end{equation}
where the maximum is taken over all affine lines $\ell\subset \mathbb{F}_q^2$ with direction $[\mathbf{v}]$.

\begin{lemma}
\label{lem:planar-domination}
Let $1\le u\leq \infty$ and 
$F:\mathbb{H}_1(\mathbb{F}_q)\to\mathbb{C}$ be given. For $(x,y)\in \Fq^2$, define
\[
G_{F,u}(x,y):=\Bigl(\sum_{t\in\mathbb{F}_q}|F(x,y,t)|^u\Bigr)^{\frac{1}{u}}
\]
when $1\leq u<\infty$, and 
\[
G_{F,\infty}(x,y):=\max\limits_{t\in\mathbb{F}_q}|F(x,y,t)|
\]
when $u=\infty$. Then
\begin{equation}\label{eq:Gu-norm}
\|G_{F,u}\|_{\ell^u(\mathbb{F}_q^2)}= \|F\|_{\ell^u(\mathbb{H}_1(\mathbb{F}_q))},
\end{equation}
and
\begin{equation}\label{eq:MH-by-Mpl}
\mathcal{M}_{\mathbb{H}_1} F([\mathbf{v}]) \le \mathcal{M}_{2}G_{F,u}([\mathbf{v}])
\end{equation}
for all $[\mathbf{v}]\in\mathbb{P}^1(\mathbb{F}_q)$. 
\end{lemma}

\begin{proof}
For given $F$ and $u$, write $G:=G_{F,u}$ for simplicity. Below we only show details of the proof for the case $1\leq u<\infty$. Similar arguments work for the case $u=\infty$. 

The identity \eqref{eq:Gu-norm} follows by expanding both sides:
\[
\|G\|_{\ell^u(\mathbb{F}_q^2)}^u
=
\sum_{x,y \in \mathbb{F}_q}\sum_{t \in \mathbb{F}_q} |F(x,y,t)|^u
=
\|F\|_{\ell^u(\mathbb{H}_1(\mathbb{F}_q))}^u.
\]

Fix $[\mathbf{v}]\in\mathbb{P}^1(\mathbb{F}_q)$ and choose a representative $\mathbf{v}=(a,b)\neq (0,0)$. Fix $\mathbf{p}=(x_0,y_0,t_0)$ and set $L=L_{\mathbf{p}, [\mathbf{v}]}$. Let $\pi:\mathbb H_1(\mathbb F_q)\to\mathbb F_q^2$ be the projection defined by $\pi(x,y,t)=(x,y)$. 
By \eqref{eq:line-param-main}, the projected set $\pi(L)$ equals the affine line
\[
\ell=\{(x_0+as,\ y_0+bs):s\in\mathbb{F}_q\}\subset \mathbb{F}_q^2,
\]
which has direction $[\mathbf{v}]$. Since $\mathbf{v}\neq 0$, the map $s\mapsto (x_0+as,y_0+bs)$ is a bijection from $\mathbb{F}_q$ onto $\ell$.
Therefore, for each $(x,y)\in \ell$, there is a unique $t=t(x,y)$ such that $(x,y,t)\in L$. Hence
\[
\sum_{\mathbf{z}\in L}|F(\mathbf{z})|
=
\sum_{(x,y)\in \ell} |F(x,y,t(x,y))|.
\]
For each fixed $(x,y)$, one has
\[
|F(x,y,t(x,y))|
\le
\Bigl(\sum_{t\in\mathbb{F}_q}|F(x,y,t)|^u\Bigr)^{\frac{1}{u}}
=
G(x,y).
\]
Summing over $(x,y)\in\ell$ implies
\[
\sum_{\mathbf{z}\in L}|F(\mathbf{z})|
\le
\sum_{(x,y)\in \ell} G(x,y)
\le
\mathcal{M}_{2}G([\mathbf{v}]).
\]
Taking the maximum over $\mathbf{p}$ gives \eqref{eq:MH-by-Mpl}. 
\end{proof}

\begin{remark}\label{goodremark}
It is worth noting that the projection domination $\mathcal M_{\mathbb{H}_1} F\le \mathcal M_2G_{F,u}$ is not reversible for any $1\leq u\leq \infty$. To see this, let $q$ be odd and fix a nonsquare $\eta\in\mathbb F_q^*$ and define
\[
\phi:\mathbb F_q^2\to\mathbb F_q,
\qquad
\phi(x,y):=x^2-\eta y^2.
\]
Let
\[
E_\phi:=\{(x,y,t)\in\mathbb H_1(\mathbb F_q):\ t=\phi(x,y)\}.
\]
Setting $F=\mathbf 1_{E_\phi}$, we have $G_{F,u}=1$ for any $1\leq u\leq \infty$. For all directions $[\mathbf{v}]$, one verifies that $\mathcal M_2G([\mathbf v])= q$ and $\mathcal M_{\mathbb{H}_1} F([\mathbf v])\le 2$, since every horizontal line meets $E_\phi$ in at most two points. 
\end{remark}

\subsection{Planar Kakeya $\ell^2$ bound}

\begin{lemma}\label{lem:planar-l2}
For every function $f:\mathbb{F}_q^2\to\mathbb{C}$, one has
\[
\|\mathcal{M}_{2}f\|_{\ell^2(\mathbb{P}^{1}(\mathbb{F}_q))}
\le
\sqrt{2}\,q^{\frac{1}{2}}\,\|f\|_{\ell^2(\mathbb{F}_q^2)}.
\]
\end{lemma}

\begin{proof}
Let $\widetilde{f}:=|f|$. 
For every affine line $\ell\subset\mathbb{F}_q^2$, one has
\[
\sum_{\mathbf x\in\ell}|f(\mathbf x)|
=
\sum_{\mathbf x\in\ell}\tilde{f}(\mathbf x).
\]
Taking the maximum over all affine lines $\ell$ with direction $[\mathbf{v}]$ implies
\[
\mathcal{M}_{2}f([\mathbf{v}])=\mathcal{M}_{2}\tilde{f}([\mathbf{v}])
\qquad\text{for all }[\mathbf{v}]\in\mathbb{P}^1(\mathbb{F}_q).
\]
Also,
\[
\|\tilde{f}\|_{\ell^2(\mathbb{F}_q^2)}=\|f\|_{\ell^2(\mathbb{F}_q^2)}.
\]
Therefore, it is enough to prove the stated inequality for $\tilde{f}$. 
We may assume without loss of generality that $f\ge 0$.

For each direction $[\mathbf{v}]\in\mathbb{P}^1(\mathbb{F}_q)$, we choose an affine line
$\ell_{[\mathbf{v}]}\subset\mathbb{F}_q^2$ with direction $[\mathbf{v}]$. 
For this fixed family $\{\ell_{[\mathbf{v}]}\}_{[\mathbf{v}]}$, define the linear operator
$T:\ell^2(\mathbb{F}_q^2)\to \ell^2(\mathbb{P}^1(\mathbb{F}_q))$ by
\[
Th([\mathbf{v}]):=\sum_{\mathbf x\in\ell_{[\mathbf{v}]}} h(\mathbf x).
\]
It is enough to prove that the estimate 
\begin{equation}\label{eq111}
\|T\|_{\ell^2(\mathbb{F}_q^2)\to \ell^2(\mathbb{P}^1(\mathbb{F}_q))}\le \sqrt{2q},
\end{equation}
holds uniformly over all such line families. Here $\|T\|_{X\to Y}:=\sup\{\|Th\|_Y:\ \|h\|_X=1\}$ denotes the operator norm. 

To see how this implies the conclusion, for a given function $f$, and each $[\mathbf{v}]\in\mathbb{P}^1(\mathbb{F}_q)$, we choose an affine line
$\ell_{[\mathbf{v}]}\subset\mathbb{F}_q^2$ with direction $[\mathbf{v}]$ such that
\[
\mathcal{M}_{2}f([\mathbf{v}])=\sum_{\mathbf x\in \ell_{[\mathbf{v}]}} f(\mathbf x).
\]
Let $T$ be the operator associated with this family.
Since the bound $\|T\|_{\ell^2(\mathbb{F}_q^2)\to \ell^2(\mathbb{P}^1(\F_q))}\le \sqrt{2q}$ is uniform over all line families, it applies to this choice.
Therefore, $Tf=\mathcal{M}_{2}f$, and
\begin{equation} \label{eq_result_need}
\|\mathcal{M}_{2}f\|_{\ell^2(\mathbb{P}^1(\F_q))}
=
\|Tf\|_{\ell^2(\mathbb{P}^1(\mathbb{F}_q))}
\le
\sqrt{2q}\,\|f\|_{\ell^2(\mathbb{F}_q^2)}.
\end{equation}

We now prove the estimate (\ref{eq111}).
We compute the adjoint $T^\ast:\ell^2(\mathbb{P}^1(\mathbb{F}_q))\to \ell^2(\mathbb{F}_q^2)$.
For $g:\mathbb{P}^1(\mathbb{F}_q)\to\mathbb{C}$ and $\mathbf x\in\mathbb{F}_q^2$, one has
\[
T^\ast g(\mathbf x)=\sum_{[\mathbf{v}]\in\mathbb{P}^1(\mathbb{F}_q)} g([\mathbf{v}])\,\mathbf{1}_{\ell_{[\mathbf{v}]}}(\mathbf x).
\]
Therefore, for each $[\mathbf{v}]\in\mathbb{P}^1(\mathbb{F}_q)$,
\begin{align*}
(TT^\ast g)([\mathbf{v}])
&=
\sum_{\mathbf x \in\ell_{[\mathbf{v}]}} T^*g(\mathbf x)
=
\sum_{\mathbf x \in\ell_{[\mathbf{v}]}}\ \sum_{[\mathbf{v}']\in\mathbb{P}^1(\mathbb{F}_q)}
g([\mathbf{v}'])\,\mathbf{1}_{\ell_{[\mathbf{v}']}}(\mathbf x) \\
&=
\sum_{[\mathbf{v}']\in\mathbb{P}^1(\mathbb{F}_q)} g([\mathbf{v}'])\,|\ell_{[\mathbf{v}]}\cap \ell_{[\mathbf{v}']}|.
\end{align*}
If $[\mathbf{v}']=[\mathbf{v}]$, then $|\ell_{[\mathbf{v}]}\cap \ell_{[\mathbf{v}']}|=|\ell_{[\mathbf{v}]}|=q$.
If $[\mathbf{v}']\neq [\mathbf{v}]$, then $\ell_{[\mathbf{v}]}$ and $\ell_{[\mathbf{v}']}$ have distinct directions, hence they meet in exactly one point in $\mathbb{F}_q^2$, so $|\ell_{[\mathbf{v}]}\cap \ell_{[\mathbf{v}']}|=1$.
Thus
\begin{equation}\label{eq:TTstar}
(TT^\ast g)([\mathbf{v}])
=
q\,g([\mathbf{v}])
+
\sum_{\substack{[\mathbf{v}']\in\mathbb{P}^1(\mathbb{F}_q)\\ [\mathbf{v}']\neq [\mathbf{v}]}}
g([\mathbf{v}']).
\end{equation}
Let $J:\ell^2(\mathbb{P}^1(\mathbb{F}_q))\to \ell^2(\mathbb{P}^1(\mathbb{F}_q))$ be the rank--one operator
\[
Jg([\mathbf{v}]):=\sum_{[\mathbf{v}']\in\mathbb{P}^1(\mathbb{F}_q)} g([\mathbf{v}']).
\]
Then \eqref{eq:TTstar} may be rewritten as
\[
TT^\ast=(q-1)I+J.
\]
We now determine the eigenvalues of $J$.
Let $\mathbf{1}$ denote the constant function $\mathbf{1}([\mathbf{v}]):=1$ on $\mathbb{P}^1(\mathbb{F}_q)$.
Then
\[
J\mathbf{1}([\mathbf{v}])=\sum_{[\mathbf{v}']\in\mathbb{P}^1(\mathbb{F}_q)} 1
=
|\mathbb{P}^1(\mathbb{F}_q)|
=
q+1,
\]
so $\mathbf{1}$ is an eigenvector of $J$ with eigenvalue $q+1$.

Next, consider the orthogonal complement of $\mathrm{span}\{\mathbf{1}\}$ in
$\ell^2(\mathbb{P}^1(\mathbb{F}_q))$, namely
\[
\Bigl(\mathrm{span}\{\mathbf{1}\}\Bigr)^\perp
=
\Bigl\{g:\mathbb{P}^1(\mathbb{F}_q)\to\mathbb{C}:\ \langle g,\mathbf{1}\rangle=0\Bigr\}.
\]
Since the inner product is $\langle g,h\rangle=\sum\limits_{[\mathbf{v}]} g([\mathbf{v}])\overline{h([\mathbf{v}])}$, the condition
$\langle g,\mathbf{1}\rangle=0$ is equivalent to
\[
\sum_{[\mathbf{v}]\in\mathbb{P}^1(\mathbb{F}_q)} g([\mathbf{v}])=0.
\]
For such $g$, one has $Jg\equiv 0$, hence every $g\in(\mathrm{span}\{\mathbf{1}\})^\perp$ is an eigenvector of $J$ with eigenvalue $0$.

Therefore, the spectrum of $J$ consists of the eigenvalue $q+1$ on $\mathrm{span}\{\mathbf{1}\}$ and the eigenvalue $0$ on its orthogonal complement.
It follows that the spectrum of $TT^*=(q-1)I+J$ consists of exactly two eigenvalues
\[
(q-1)+(q+1)=2q
\qquad\text{and}\qquad
(q-1)+0=q-1.
\]
Hence,
\[
\|TT^\ast\|_{\ell^2(\mathbb{P}^1(\mathbb{F}_q))\to \ell^2(\mathbb{P}^1(\mathbb{F}_q))}=2q.
\]
Consequently,
\[
\|T\|_{\ell^2(\mathbb{F}_q^2)\to \ell^2(\mathbb{P}^1(\mathbb{F}_q))}^2
=
\|TT^\ast\|_{\ell^2(\mathbb{P}^1(\mathbb{F}_q))\to \ell^2(\mathbb{P}^1(\mathbb{F}_q))}
=
2q,
\]
so $\|T\|_{\ell^2(\mathbb{F}^2_q)\to \ell^2(\mathbb{P}^1(\F_q))}\le \sqrt{2q}$. Now \eqref{eq_result_need} follows, and the proof is completed. 
\end{proof}

\section{Upper bound for the diagonal (Theorem \ref{thm:exact_A})}\label{section3}

\subsection{Planar endpoint bounds}

\begin{lemma}
\label{lem:planar-endpoints}
For all $G:\mathbb{F}_q^2\to[0,\infty)$, we have
\begin{align}
\|\mathcal{M}_{2}G\|_{\ell^1(\mathbb{P}^1(\mathbb{F}_q))} &\le (q+1)\,\|G\|_{\ell^1(\mathbb{F}_q^2)},\label{eq:pl-11}\\
\|\mathcal{M}_{2}G\|_{\ell^2(\mathbb{P}^1(\mathbb{F}_q))} &\le \sqrt{2}\,q^{\frac{1}{2}}\,\|G\|_{\ell^2(\mathbb{F}_q^2)},\label{eq:pl-22}\\
\|\mathcal{M}_{2}G\|_{\ell^\infty(\mathbb{P}^1(\mathbb{F}_q))} &\le q\,\|G\|_{\ell^\infty(\mathbb{F}_q^2)}.\label{eq:pl-infty}
\end{align}
\end{lemma}

\begin{proof}
For \eqref{eq:pl-11}, note that for every direction class $[\mathbf{v}]$, one has
\[
\mathcal{M}_{2} G([\mathbf{v}])
    \le \sum_{\mathbf{z}\in\mathbb{F}_q^{\,2}} G(\mathbf{z})
    = \|G\|_{\ell^{1}(\mathbb{F}_q^{2})}.
\]
Summing over all $[\mathbf{v}]$ immediately implies \eqref{eq:pl-11}. For \eqref{eq:pl-infty}, observe that each affine line in $\mathbb{F}_q^{\,2}$ contains exactly $q$ points, and hence
\[
\sum_{\mathbf{z}\in\ell} G(\mathbf{z})
    \le q\,\|G\|_{\ell^{\infty}(\mathbb{F}_q^{2})}
    \qquad \text{for every line } \ell \subset \mathbb{F}_q^2.
\]
Taking the maximum over all lines $\ell$ gives \eqref{eq:pl-infty}. Finally, \eqref{eq:pl-22} follows directly from Lemma~\ref{lem:planar-l2}.
\end{proof}

\subsection{A planar $\ell^u\to\ell^u$ theorem}

\begin{theorem}
\label{thm:planar-all-u}
For every $1\le u\le\infty$ and for all $G:\mathbb{F}_q^2\to[0,\infty)$, we have 
\[
\|\mathcal{M}_{2}G\|_{\ell^u(\mathbb{P}^1(\mathbb{F}_q))}
\le \sqrt{2}
\,q^{\tau(u)}\,\|G\|_{\ell^u(\mathbb{F}_q^2)}
\]
where
\[
\tau(u):=
\begin{cases}
\displaystyle \frac1u, & 1\le u\le 2,\\[0.4em]
\displaystyle 1-\frac1u, & 2\le u\le \infty.
\end{cases}
\]
\end{theorem}

 \begin{proof}
Let $G:\mathbb{F}_q^2\to[0,\infty)$ be given.
For each direction $[\mathbf{v}]\in\mathbb{P}^1(\mathbb{F}_q)$, choose an affine line
$\ell_{[\mathbf{v}]}\subset\mathbb{F}_q^2$ with direction $[\mathbf{v}]$ such that
\[
\mathcal{M}_{2}G([\mathbf{v}])=\sum_{\mathbf{x} \in \ell_{[\mathbf{v}]}} G(\mathbf{x}).
\]
Define the linear operator $T$ on functions $f:\mathbb{F}_q^2\to\mathbb{C}$ by
\[
(T f)([\mathbf{v}]):=\sum_{\mathbf{x} \in\ell_{[\mathbf{v}]}} f(\mathbf{x}),
\qquad [\mathbf{v}]\in\mathbb{P}^1(\mathbb{F}_q).
\]
Then
\begin{equation}\label{eq:TG-linearizes}
(T G)([\mathbf{v}])=\mathcal{M}_{2}G([\mathbf{v}])
\qquad\text{for all }[\mathbf{v}]\in\mathbb{P}^1(\mathbb{F}_q).
\end{equation}
Therefore,
\begin{equation}\label{eq:reduce-to-TG}
\|\mathcal{M}_{2}G\|_{\ell^u(\mathbb{P}^1(\mathbb{F}_q))}
=
\|T G\|_{\ell^u(\mathbb{P}^1(\mathbb{F}_q))}.
\end{equation}

We next record three endpoint bounds for $T$.

\medskip
\noindent\textbf{The $\ell^1$ bound.}
For every $f:\mathbb{F}_q^2\to\mathbb{C}$,
\begin{align}\label{eq:TG-l1}
\|T f\|_{\ell^1(\mathbb{P}^1(\mathbb{F}_q))}
&=
\sum_{[\mathbf{v}]\in\mathbb{P}^1(\mathbb{F}_q)}
\Bigl|\sum_{\mathbf{x} \in\ell_{[\mathbf{v}]}} f(\mathbf{x})\Bigr|
\le
\sum_{[\mathbf{v}]}\sum_{\mathbf{x} \in\ell_{[\mathbf{v}]}}|f(\mathbf{x})|
\le
(q+1)\sum_{\mathbf{x} \in\mathbb{F}_q^2}|f(\mathbf{x})| \notag\\
&=
(q+1)\,\|f\|_{\ell^1(\mathbb{F}_q^2)}.
\end{align}

\medskip
\noindent\textbf{The $\ell^\infty$ bound.}
For every $f:\mathbb{F}_q^2\to\mathbb{C}$ and every $[\mathbf{v}]\in \mathbb{P}^1(\mathbb{F}_q)$,
\[
|(T f)([\mathbf{v}])|
=
\Bigl|\sum_{\mathbf{x} \in\ell_{[\mathbf{v}]}} f(\mathbf{x})\Bigr|
\le
\sum_{\mathbf{x} \in\ell_{[\mathbf{v}]}}|f(\mathbf{x})|
\le
q\,\|f\|_{\ell^\infty(\mathbb{F}_q^2)},
\]
since each affine line in $\mathbb{F}_q^2$ has exactly $q$ points.
Taking the maximum over $[\mathbf{v}]\in \mathbb{P}^1(\mathbb{F}_q)$ implies
\begin{equation}\label{eq:TG-linfty}
\|T f\|_{\ell^\infty(\mathbb{P}^1(\mathbb{F}_q))}
\le
q\,\|f\|_{\ell^\infty(\mathbb{F}_q^2)}.
\end{equation}

\medskip
\noindent\textbf{The $\ell^2$ bound.}
The operator $T$ is obtained by selecting one affine line in each direction,
so the proof of Lemma~\ref{lem:planar-l2} applies verbatim to $T$.
Therefore, for every $f:\mathbb{F}_q^2\to\mathbb{C}$,
\begin{equation}\label{eq:TG-l2}
\|T f\|_{\ell^2(\mathbb{P}^1(\mathbb{F}_q))}
\le
\sqrt{2}\,q^{\frac12}\,\|f\|_{\ell^2(\mathbb{F}_q^2)}.
\end{equation}

We now interpolate the bounds to obtain the desired exponents.

\medskip
\noindent\textbf{Case 1: $1\le u\le 2$.}
Choose $\theta\in[0,1]$ so that
\[
\frac1u=(1-\theta)\cdot 1+\theta\cdot\frac12.
\]
Then $\theta=2\bigl(1-\frac1u\bigr)$.
Applying Lemma~\ref{interpolation-inequality} to the linear operator $T$,
using the endpoint estimates \eqref{eq:TG-l1} and \eqref{eq:TG-l2}, we obtain
{ 
\begin{align*}
\|Tf\|_{\ell^u(\mathbb{P}^1(\mathbb{F}_q))}
\le
(q+1)^{1-\theta}\bigl(\sqrt{2}\,q^{\frac12}\bigr)^{\theta}
\|f\|_{\ell^u(\mathbb{F}_q^2)} 
\le
\left(\frac43\right)^{1-\theta}
2^{\frac{\theta}{2}}
q^{1-\frac{\theta}{2}}
\|f\|_{\ell^u(\mathbb{F}_q^2)},
\end{align*}
where the last inequality follows from the assumption that $q\ge 3$.

Since $1\le u\le2$, we have $\theta\in[0,1]$. Moreover,
\[
\left(\frac43\right)^{1-\theta}2^{\frac{\theta}{2}}
=
\frac43\left(\frac{3\sqrt2}{4}\right)^\theta,
\]
which is an increasing function of $\theta$ because $\frac{3\sqrt2}{4}>1$.
Therefore,
\begin{align}
\label{eq:TG-u-1-2}
\|Tf\|_{\ell^u(\mathbb{P}^1(\mathbb{F}_q))}
\le
\left(\frac43\right)^{1-\theta} \, 
2^{\frac{\theta}{2}} \, 
q^{1-\frac{\theta}{2}} \, 
\|f\|_{\ell^u(\mathbb{F}_q^2)}
\le
\sqrt{2}\,
q^{1-\frac{\theta}{2}} \, 
\|f\|_{\ell^u(\mathbb{F}_q^2)}.
\end{align}
}
From here, we apply \eqref{eq:TG-u-1-2} with $f=G$ and use \eqref{eq:reduce-to-TG}.

\medskip
\noindent\textbf{Case 2: $2\le u\le \infty$.}
Choose $\theta\in[0,1]$ so that
\[
\frac1u=(1-\theta)\cdot\frac12+\theta\cdot 0.
\]
Then $\theta=1-\frac{2}{u}$.
Applying Lemma~\ref{interpolation-inequality} to the linear operator $T$,
using the endpoint estimates \eqref{eq:TG-l2} and \eqref{eq:TG-linfty}, we obtain 
\[
\|T f\|_{\ell^u(\mathbb{P}^1(\mathbb{F}_q))}
\le
\bigl(\sqrt{2}q^{\frac12}\bigr)^{1-\theta}\,q^{\theta}\,
\|f\|_{\ell^u(\mathbb{F}_q^2)}
\leq \sqrt{2}
\,q^{\frac{1-\theta}{2}+\theta}\,\|f\|_{\ell^u(\mathbb{F}_q^2)}.
\]
Since $\frac{1-\theta}{2}+\theta=1-\frac1u$, this implies
\[
\|T f\|_{\ell^u(\mathbb{P}^1(\mathbb{F}_q))}
\le \sqrt{2} \,q^{1-\frac1u}\,\|f\|_{\ell^u(\mathbb{F}_q^2)}.
\]

Combining the two cases gives
\[
\|T f\|_{\ell^u(\mathbb{P}^1(\mathbb{F}_q))}
\le \sqrt{2}  \,q^{\tau(u)}\,\|f\|_{\ell^u(\mathbb{F}_q^2)}.
\]
Taking $f=G$ and combining \eqref{eq:reduce-to-TG}, the theorem then follows.
\end{proof}

\subsection{Diagonal bounds for $\mathcal{M}_{\mathbb{H}_1}$}

\begin{theorem}
\label{thm:diag}
For every $1\le u\le \infty$ and for all $F:\mathbb{H}_1(\mathbb{F}_q)\to \mathbb{C}$, we have 
\[
\|\mathcal{M}_{\mathbb{H}_1} F\|_{\ell^u(\mathbb{P}^1(\mathbb{F}_q))}
\le  \sqrt{2} \, q^{\tau(u)}\, 
\|F\|_{\ell^u(\mathbb{H}_1(\mathbb{F}_q))}
\]
 where $\tau(u)$ is as in Theorem~\ref{thm:planar-all-u}.
\end{theorem}

\begin{proof}
For $1\le u\leq \infty$, define $G:=G_{F,u}$ as in Lemma~\ref{lem:planar-domination}. It follows by \eqref{eq:Gu-norm}, \eqref{eq:MH-by-Mpl} and Theorem~\ref{thm:planar-all-u},
\[
\|\mathcal{M}_{\mathbb{H}_1} F\|_{\ell^u(\mathbb{P}^1(\mathbb{F}_q))}
\le
\|\mathcal{M}_{2}G\|_{\ell^u(\mathbb{P}^1(\mathbb{F}_q))}
\le \sqrt{2} \,q^{\tau(u)}\,\|G\|_{\ell^u(\mathbb{F}_q^2)}=\sqrt{2} \, q^{\tau(u)}\, 
\|F\|_{\ell^u(\mathbb{H}_1(\mathbb{F}_q))}.
\]
\end{proof}

\section{Upper bounds away from the diagonal (Theorem \ref{thm:exact_A})}\label{section4}

We now prove regionwise bounds which, when assembled, provide the upper bound in Theorem~\ref{thm:exact_A}.

\subsection{The case $A_1(u, v)\le 1+\frac{1}{v}-\frac{2}{u}$}

\begin{lemma}\label{thm:upper-left}
Let $2\le u\le \infty$ and $1\le v\le u$. Then, for all
$F:\mathbb{H}_1(\mathbb{F}_q)\to\mathbb{C}$, 
\[
\|\mathcal{M}_{\mathbb{H}_1} F\|_{\ell^v(\mathbb{P}^1(\mathbb{F}_q))}
\le 2\sqrt{2}
\,q^{\,1+\frac1v-\frac{2}{u}}\,
\|F\|_{\ell^u(\mathbb{H}_1(\mathbb{F}_q))}.
\]
\end{lemma}

\begin{proof}
Since $v\le u$, Lemma~\ref{lem:embed} gives
\[
\|\mathcal{M}_{\mathbb{H}_1} F\|_{\ell^v(\mathbb{P}^1(\mathbb{F}_q))}
\le
|\mathbb{P}^1(\mathbb{F}_q)|^{\frac1v-\frac1u}\,\|\mathcal{M}_{\mathbb{H}_1} F\|_{\ell^u(\mathbb{P}^1(\mathbb{F}_q))}.
\]
Since $u\ge 2$, Theorem~\ref{thm:diag} implies
\[
\|\mathcal{M}_{\mathbb{H}_1} F\|_{\ell^u(\mathbb{P}^1(\mathbb{F}_q))}
\le \sqrt{2}
\,q^{1-\frac1u}\,\|F\|_{\ell^u(\mathbb{H}_1(\mathbb{F}_q))}.
\]
Therefore,
\[
\|\mathcal{M}_{\mathbb{H}_1} F\|_{\ell^v(\mathbb{P}^1(\mathbb{F}_q))}
\le \sqrt{2} \,(q+1)^{\frac1v-\frac1u}\,q^{1-\frac1u}\,
\|F\|_{\ell^u(\mathbb{H}_1(\mathbb{F}_q))}
\le 2\sqrt{2}\,q^{\,1+\frac1v-\frac{2}{u}}\,
\|F\|_{\ell^u(\mathbb{H}_1(\mathbb{F}_q))}.
\]
This completes the proof.
\end{proof}

\subsection{The case $A_1(u, v)\le 1-\frac{1}{u}$}

\begin{lemma}\label{thm:blue}
Let $2\le u\le v\le\infty$. Then, for all
$F:\mathbb{H}_1(\mathbb{F}_q)\to\mathbb{C}$,  
\[
\|\mathcal{M}_{\mathbb{H}_1} F\|_{\ell^v(\mathbb{P}^1(\mathbb{F}_q))}
\le  \sqrt{2} \,q^{1-\frac{1}{u}}\,
\|F\|_{\ell^u(\mathbb{H}_1(\mathbb{F}_q))}.
\]
\end{lemma}

\begin{proof}
Since $u\le v$, applying Lemma~\ref{lem:embed} with $g=\mathcal{M}_{\mathbb{H}_1} F$ gives
\[
\|\mathcal{M}_{\mathbb{H}_1} F\|_{\ell^v(\mathbb{P}^1(\mathbb{F}_q))}
\le
\|\mathcal{M}_{\mathbb{H}_1} F\|_{\ell^u(\mathbb{P}^1(\mathbb{F}_q))}.
\]
Since $u\ge 2$, Theorem~\ref{thm:diag} implies
\[
\|\mathcal{M}_{\mathbb{H}_1} F\|_{\ell^u(\mathbb{P}^1(\mathbb{F}_q))}
\le \sqrt{2} \,q^{1-\frac{1}{u}}\,
\|F\|_{\ell^u(\mathbb{H}_1(\mathbb{F}_q))}.
\]
Thus the desired bound follows.
\end{proof}

\begin{lemma}\label{thm:lower-right-1}
Let $1\le u\le 2$ and $v\ge \frac{u}{u-1}$. Then,
for all
$F:\mathbb{H}_1(\mathbb{F}_q)\to\mathbb{C}$,
\[
\|\mathcal{M}_{\mathbb{H}_1} F\|_{\ell^v(\mathbb{P}^1(\mathbb{F}_q))}
\le \sqrt{2}\,q^{1-\frac1u}\,
\|F\|_{\ell^u(\mathbb{H}_1(\mathbb{F}_q))}.
\]
\end{lemma}

\begin{proof}
As before, it is sufficient to prove the claimed bound for non-negative $F$.

For each $[\mathbf{v}]\in\mathbb{P}^1(\mathbb{F}_q)$, choose a horizontal line $L_{[\bfv]}\in \cL([\bfv])$ such that 
\[
\mathcal{M}_{\mathbb{H}_1} F([\mathbf{v}])=\sum_{\mathbf{p}\in L_{[\mathbf{v}]}} F(\mathbf{p}).
\] 
With this family $\{L_{[\mathbf{v}]}\}_{[\mathbf{v}]}$ of lines, we define the associated linear transformation $T$ on functions $\mathcal{G}:\mathbb{H}_1(\mathbb{F}_q)\to\mathbb{C}$ by
\[
(T\mathcal{G})([\mathbf{v}]):= \sum_{\mathbf{p}\in L_{[\mathbf{v}]}} \mathcal{G}(\mathbf{p}).
\]
It is sufficient to prove that 
\[
\|T \mathcal{G}\|_{\ell^v(\mathbb{P}^1(\mathbb{F}_q))}
\le \sqrt{2}\,q^{1-\frac1u}\,
\|\mathcal{G}\|_{\ell^u(\mathbb{H}_1(\mathbb{F}_q))}.
\]
Then the lemma follows by taking $\mathcal G=F$.

We record two bounds for $T$.
For every $\mathcal{G}:\mathbb{H}_1(\mathbb{F}_q)\to\mathbb{C} $ and every $[\mathbf{v}]$,
\[
|(T\mathcal{G})([\mathbf{v}])|
\le
\sum_{\mathbf{p} \in\mathbb{H}_1(\mathbb{F}_q)} |\mathcal{G}(\mathbf{p})|.
\]
Therefore,
\begin{equation}\label{eq:lr1-1-infty}
\|T\mathcal{G}\|_{\ell^\infty(\mathbb{P}^1(\mathbb{F}_q))}
\le
\|\mathcal{G}\|_{\ell^1(\mathbb{H}_1(\mathbb{F}_q))}.
\end{equation}
Also, the diagonal estimate at $(2,2)$ applies to the chosen family of lines, so 
\begin{equation}\label{eq:lr1-2-2}
\|T\mathcal{G}\|_{\ell^2(\mathbb{P}^1(\mathbb{F}_q))}
\le \sqrt{2} \,q^{\frac12}\,\|\mathcal{G}\|_{\ell^2(\mathbb{H}_1(\mathbb{F}_q))}
\qquad\text{for all } \mathcal{G}:\mathbb{H}_1(\mathbb{F}_q)\to\mathbb{C}.
\end{equation}
From \eqref{eq:lr1-1-infty} and \eqref{eq:lr1-2-2}, we apply Lemma~\ref{interpolation-inequality} to obtain the desired inequality. 
To this end, we observe that \((p_0,r_0) = (1,\infty)\) and \((p_1,r_1) = (2,2)\).
We now choose a parameter \(\theta \in [0,1]\) such that
\[
\frac{1}{u} = \frac{1-\theta}{p_0} + \frac{\theta}{p_1},
    \qquad
\frac{1}{r} = \frac{1-\theta}{r_0} + \frac{\theta}{r_1}.
\]
Solving this system implies 
\[
r = \frac{2}{\theta} = \frac{u}{u-1},
    \qquad 
\theta = 2\Bigl(1 - \frac{1}{u}\Bigr).
\]
With these values, combining \eqref{eq:lr1-1-infty}, \eqref{eq:lr1-2-2}, and Lemma~\ref{interpolation-inequality} gives
\begin{align*}
   \|T\mathcal{G}\|_{\ell^{v}(\mathbb{P}^{1}(\mathbb{F}_q))}
   \le 
   \|T\mathcal{G}\|_{\ell^{\frac{u}{u-1}}(\mathbb{P}^{1}(\mathbb{F}_q))} 
   \le \sqrt{2}\, q^{\frac{\theta}{2}}\,
   \|\mathcal{G}\|_{\ell^{u}(\mathbb{H}_{1}(\mathbb{F}_q))} 
   =\sqrt{2}\,q^{1-\frac1u}\,
\|\mathcal{G}\|_{\ell^{u}(\mathbb{H}_1(\mathbb{F}_q))},
\end{align*}
where we have used Lemma~\ref{lem:embed} noting that $v \geq \frac{u}{u-1}$ in the first inequality and $\frac{\theta}{2}=1-\frac1u$ in the last equality. 
\end{proof}

\subsection{The case $A_1(u, v)\le \frac{1}{v}$}

\begin{lemma}\label{thm:green}
Let $1\le v\le u\le 2$. Then, for all
$F:\mathbb{H}_1(\mathbb{F}_q)\to\mathbb{C}$, 
\[
\|\mathcal{M}_{\mathbb{H}_1} F\|_{\ell^v(\mathbb{P}^1(\mathbb{F}_q))}
\le 2 \,q^{\frac{1}{v}}\,
\|F\|_{\ell^u(\mathbb{H}_1(\mathbb{F}_q))}.
\]
\end{lemma}

\begin{proof}
Since $v\le u$, Lemma~\ref{lem:embed} gives
\[
\|\mathcal{M}_{\mathbb{H}_1} F\|_{\ell^v(\mathbb{P}^1(\mathbb{F}_q))}
\le
|\mathbb{P}^1(\mathbb{F}_q)|^{\frac1v-\frac1u}\,
\|\mathcal{M}_{\mathbb{H}_1} F\|_{\ell^u(\mathbb{P}^1(\mathbb{F}_q))}.
\]
Since $u\le 2$, Theorem~\ref{thm:diag} implies
\[
\|\mathcal{M}_{\mathbb{H}_1} F\|_{\ell^u(\mathbb{P}^1(\mathbb{F}_q))}
\le \sqrt{2} \,q^{\frac{1}{u}}\,
\|F\|_{\ell^u(\mathbb{H}_1(\mathbb{F}_q))}.
\]
Therefore,
\begin{align*}
    \|\mathcal{M}_{\mathbb{H}_1} F\|_{\ell^v(\mathbb{P}^1(\mathbb{F}_q))}
\le & \sqrt{2}\,(q+1)^{\frac1v-\frac1u}q^{\frac{1}{u}}\,\|F\|_{\ell^u(\mathbb{H}_1(\mathbb{F}_q))} \leq \sqrt{2}\,(2q)^{\frac1v-\frac1u}q^{\frac{1}{u}}\,\|F\|_{\ell^u(\mathbb{H}_1(\mathbb{F}_q))} \\
\le & 2\,q^{\frac{1}{v}}\,\|F\|_{\ell^u(\mathbb{H}_1(\mathbb{F}_q))},
\end{align*}
where the last inequality uses the fact that $0\leq \frac{1}{v}-\frac{1}{u} \leq \frac{1}{2}$ 
which follows from the assumption \(1 \le v \le u \le 2\).
This completes the proof.
\end{proof}

\begin{lemma}
\label{thm:lower-right-2}
Let $1\le u\le 2$ and $u\le v\le \frac{u}{u-1}$ (with the convention $\frac{u}{u-1}=\infty$ when $u=1$).  Then, for all
$F:\mathbb{H}_1(\mathbb{F}_q)\to\mathbb{C}$, 
\[
\|\mathcal{M}_{\mathbb{H}_1} F\|_{\ell^v(\mathbb{P}^1(\mathbb{F}_q))}
\le 2\sqrt{2}\,q^{\frac1v}\,
\|F\|_{\ell^u(\mathbb{H}_1(\mathbb{F}_q))}.
\]
\end{lemma}

\begin{proof}
It suffices to prove the claim for non-negative $F$. We take the family $\{L_{[\mathbf{v}]}\}_{[\mathbf{v}]}$ of horizontal lines and the linear operator $T$ as in the proof of Lemma \ref{thm:lower-right-1}.

We record two endpoint bounds for $T$. First, we deduce by \eqref{eq:lr1-1-infty} that 
\begin{equation}\label{eq:T_1_to_r}
\|T\mathcal{G}\|_{\ell^r(\mathbb{P}^1(\mathbb{F}_q))}\leq (q+1)^{\frac1r} \|T\mathcal{G}\|_{\ell^\infty( \mathbb{P}^1(\mathbb{F}_q))}
\le (q+1)^{\frac{1}{r}}\,\|\mathcal{G}\|_{\ell^1(\mathbb{H}_1(\mathbb{F}_q))}.
\end{equation}
Second, since $|T\mathcal{G}|\le \mathcal{M}_{\mathbb{H}_1}(|\mathcal{G}|)$ pointwise, Theorem~\ref{thm:diag} with $u=2$ gives
\begin{equation}\label{eq:T_2_to_2}
\|T\mathcal{G}\|_{\ell^2(\mathbb{P}^1(\mathbb{F}_q))}
\le \|\mathcal{M}_{\mathbb{H}_1}(|\mathcal{G}|)\|_{\ell^2(\mathbb{P}^1(\mathbb{F}_q))}
\le \sqrt{2}\, q^{\frac{1}{2}}\, \|\mathcal{G}\|_{\ell^2(\mathbb{H}_1(\mathbb{F}_q))}.
\end{equation}
If $u=2$, then $v=2$ and \eqref{eq:T_2_to_2} gives the desired estimate.

Assume now that $1\le u<2$. Set
\[
\theta := 2\Bigl(1-\frac1u\Bigr)\in[0,1),
\qquad\text{so that}\qquad
\frac1u=(1-\theta)\cdot 1 + \theta\cdot \frac12.
\]
Choose $r\in[1,\infty]$ so that
\begin{align}\label{definition-r-v}
    \frac1v = (1-\theta)\cdot \frac1r + \theta\cdot \frac12.
\end{align}
Such an $r$ exists because $u\le v\le \frac{u}{u-1}$ is equivalent to
\[
\frac1u \ge \frac1v \ge 1-\frac1u = \frac{\theta}{2},
\]
and hence $0\le \frac1r \le 1$.

Apply Lemma~\ref{interpolation-inequality} to the linear operator $T$ using \eqref{eq:T_1_to_r} and \eqref{eq:T_2_to_2}.
We obtain
\[
\|T\mathcal{G}\|_{\ell^v(\mathbb{P}^{1}(\mathbb{F}_q))}
\le
\bigl((q+1)^{\frac{1}{r}}\bigr)^{1-\theta}\,
\bigl(\sqrt{2}\,q^{\frac{1}{2}}\bigr)^{\theta}\,
\|\mathcal{G}\|_{\ell^u(\mathbb{H}_1(\mathbb{F}_q))} \leq 2\sqrt{2}\, q^{\frac{1}{v}}\, \|\mathcal{G}\|_{\ell^u(\mathbb{H}_1(\mathbb{F}_q))}.
\]
where we have used $q+1\leq 2q$ and
\eqref{definition-r-v} in the last inequality. Now the lemma follows by taking $\mathcal G=F$.
\end{proof}

\section{Lower bounds (Theorem \ref{thm:exact_A})}\label{section5}

\subsection{The case $A_1(u, v)\ge 1+\frac{1}{v}-\frac{2}{u}$}

\begin{lemma}
\label{lem:lb-fan}
For every $1\le u,v\le\infty$, one has $A_1(u, v)\ge 1+\frac1v-\frac{2}{u}$.
\end{lemma}

\begin{proof}
Let
\[
S:=\bigcup_{[\mathbf{v}]\in\mathbb{P}^1(\mathbb{F}_q)} L_{\mathbf{0}, [\mathbf{v}]},
\qquad
F:=\mathbf{1}_S.
\]
Distinct horizontal lines through $\mathbf{0}$ intersect only at $\mathbf{0}$.
Since each such line has $q$ points and there are $q+1$ directions, we obtain
\[
|S|=1+(q+1)(q-1)=q^2.
\]
Therefore, $\|F\|_{\ell^u(\mathbb{H}_1(\mathbb{F}_q))}=|S|^{\frac{1}{u}}=q^{\frac{2}{u}}$.

For every $[\mathbf{v}]\in\mathbb{P}^1(\mathbb{F}_q)$, the line $L_{\mathbf{0}, [\bfv]}$ is contained in $S$, hence
\[
\mathcal{M}_{\mathbb{H}_1} F([\mathbf{v}])\ge \sum_{\mathbf x\in L_{\mathbf{0}, [\mathbf{v}]}}\mathbf{1}_S(\mathbf x)=q.
\]
Consequently,
\[
\|\mathcal{M}_{\mathbb{H}_1} F\|_{\ell^v(\mathbb{P}^1(\mathbb{F}_q))}
\ge
\Bigl(\sum_{[\mathbf{v}]\in\mathbb{P}^1(\mathbb{F}_q)} q^v\Bigr)^{\frac{1}{v}}
=
q\,(q+1)^{\frac{1}{v}}
\ge  q^{1+\frac{1}{v}}.
\]
Therefore, $q^{A_1(u, v)}\gtrsim q^{1+\frac{1}{v}}/q^{\frac{2}{u}}$, and hence $A_1(u, v)\ge 1+\frac1v-\frac{2}{u}$.
\end{proof}

\subsection{The case $A_1(u, v)\ge 1-\frac{1}{u}$}

\begin{lemma}
\label{lem:lb-line}
For every $1\le u,v\le\infty$, one has $A_1(u, v)\ge 1-\frac1u$.
\end{lemma}

\begin{proof}
Let $L\subset\mathbb{H}_1(\mathbb{F}_q)$ be a fixed horizontal line, and set $F=\mathbf{1}_L$.
Then $\|F\|_{\ell^u(\mathbb{H}_1(\mathbb{F}_q))}=|L|^{\frac{1}{u}}=q^{\frac{1}{u}}$.
Let $[\mathbf{v}_0]$ be the direction of $L$. 
Then $\mathcal{M}_{\mathbb{H}_1} F([\mathbf v_0])\ge \sum\limits_{\mathbf x\in L}\mathbf{1}_L(\mathbf x)=q$.
Therefore, $\|\mathcal{M}_{\mathbb{H}_1} F\|_{\ell^v(\mathbb{P}^1(\mathbb{F}_q))}\ge q$ for every $v$.
Thus $q^{A_1(u, v)}\gtrsim q/q^{\frac{1}{u}}$, and hence $A_1(u, v)\ge 1-\frac{1}{u}$.
\end{proof}

\subsection{The case $A_1(u, v)\ge \frac{1}{v}$}

\begin{lemma}
\label{lem:lb-point}
For every $1\le u,v\le\infty$ one has $A_1(u, v)\ge \frac1v$.
\end{lemma}

\begin{proof}
Fix a point $\bfp_\ast\in\mathbb{H}_1(\mathbb{F}_q)$ and set $F=\delta_{\bfp_\ast}$ where $\delta_{\bfp_\ast} :\mathbb{H}_1(\mathbb{F}_q) \to \{0,1\}$ is defined by 
\begin{equation*}
\delta_{\mathbf p_\ast} (\mathbf x)=
    \begin{cases}
        1, ~~~~~~~~\mathbf x=\bfp_\ast,\\
        0, ~~~~~~~~\mathbf x \ne \bfp_\ast.
    \end{cases}
\end{equation*}
For every direction $[\mathbf{v}]\in\mathbb{P}^1(\mathbb{F}_q)$ there is a unique horizontal line in direction $[\mathbf{v}]$ through $\mathbf p_\ast$.
Therefore, $\mathcal{M}_{\mathbb{H}_1} F([\mathbf{v}])=1$ for all $[\mathbf{v}]$.
Hence
\[
\|\mathcal{M}_{\mathbb{H}_1} F\|_{\ell^v(\mathbb{P}^1(\mathbb{F}_q))}=(q+1)^{\frac{1}{v}},
\qquad
\|F\|_{\ell^u(\mathbb{H}_1(\mathbb{F}_q))}=1.
\]
Thus, any admissible exponent in \eqref{eq:An_def} must satisfy $q^{A_1(u, v)}\gtrsim (q+1)^{\frac{1}{v}}$, and hence $A_1(u, v)\ge \frac{1}{v}$.

\end{proof}

\section{Proof of Theorem \ref{thm:exact_A} -- $\mathbb{H}_1(\mathbb{F}_q)$}\label{section6}

\begin{proof}[Proof of Theorem~\ref{thm:exact_A}]
It follows from Lemmas~\ref{lem:lb-point}, \ref{lem:lb-line}, and \ref{lem:lb-fan} that
\[
A_1(u, v)\ge \max\Bigg\{\frac1v,\ 1-\frac1u,\ 1+\frac1v-\frac{2}{u}\Bigg\}.
\]
It remains to prove the matching upper bound. We distinguish four regions.

\medskip
\noindent\textbf{(i) The region: $1\le v\le u\le 2$.}
By Lemma~\ref{thm:green} one has $A_1(u, v)\le \frac{1}{v}$.
In this region $\frac1u\ge \frac12$ and $\frac1v\ge \frac1u$, hence $1+\frac1v-\frac{2}{u}\le \frac1v$ and $1-\frac1u\le \frac1v$.
Therefore, $\max \big\{\frac1v,1-\frac1u,1+\frac1v-\frac{2}{u} \big\}=\frac1v$, so the upper bound is sharp.

\medskip
\noindent\textbf{(ii) The region: $2\le u\le v\le \infty$.}
By Lemma~\ref{thm:blue}, one has $A_1(u, v)\le 1-\frac1u$.
In this region $\frac1u\le \frac12$ and $\frac1v\le \frac1u$, hence $1+\frac1v-\frac{2}{u}\le 1-\frac1u$ and $\frac1v\le 1-\frac1u$.
Therefore, the maximum equals $1-\frac1u$, so the upper bound is sharp.

\medskip
\noindent\textbf{(iii) The region: $2\le u\le\infty$ and $1\le v\le u$.}
By Lemma~\ref{thm:upper-left}, one has $A_1(u, v)\le 1+\frac1v-\frac{2}{u}$.
Since $\frac1u\le \frac12$, we have $1+\frac1v-\frac{2}{u}\ge \frac1v$ and $1+\frac1v-\frac{2}{u}\ge 1-\frac1u$.
Therefore, the maximum equals $1+\frac1v-\frac{2}{u}$, so the upper bound is sharp.

\medskip
\noindent\textbf{(iv) The region: $1\le u\le 2$ and $u\le v\le \infty$.} We have $A_1(u, v)\le \max \big\{\frac1v,1-\frac1u \big\}$. Indeed, if $v\ge \frac{u}{u-1}$, apply Lemma~\ref{thm:lower-right-1}. If $u\le v\le \frac{u}{u-1}$, apply Lemma~\ref{thm:lower-right-2}.

In this region $\frac1u\ge \frac12$ and $\frac1v\le \frac1u$, hence $1+\frac1v-\frac{2}{u}\le 1-\frac1u$.
Therefore, $\max\big\{\frac1v,1-\frac1u,1+\frac1v-\frac{2}{u}\big\}=\max\big\{\frac1v,1-\frac1u \big\}$, so the upper bound matches the claimed maximum.

\medskip
Since these four regions cover all $1\le u,v\le\infty$, we conclude that
\[
A_1(u, v)= \max\Bigg\{\frac1v,\ 1-\frac1u,\ 1+\frac1v-\frac{2}{u}\Bigg\}.
\]
This completes the proof.
\end{proof}

\section{Proof of Theorem \ref{thm:exact_Hn} -- $\mathbb{H}_n(\mathbb{F}_q)$}\label{section7}

The argument is the same as that of $\mathbb{H}_1(\mathbb{F}_q)$. The proof for $\mathbb{H}_1(\mathbb{F}_q)$ is self-contained. In higher dimensions, the only new input is the diagonal estimate in $\mathbb{F}_q^{2n}$ due to Ellenberg, Oberlin, and Tao in \cite{EOT}, which is an extension of Lemma \ref{lem:planar-l2}. For completeness, we provide a sketch in this section.

Note that 
\begin{equation}\label{size-P^{2n-1}}
   \bigl|\mathbb{P}^{2n-1}(\mathbb{F}_q)\bigr|=\frac{q^{2n}-1}{q-1},
\end{equation}
and hence
\begin{equation}\label{eq:Psize}
q^{2n-1}\le \bigl|\mathbb{P}^{2n-1}(\mathbb{F}_q)\bigr|\le 2q^{2n-1}.
\end{equation}

We begin by extending the projection and domination principle of Lemma 
\ref{lem:planar-domination} to the higher-dimensional setting.

\begin{lemma}\label{lem:Hn_domination}
Let $1\le u\le \infty$ and let $F:\mathbb{H}_n(\mathbb{F}_q)\to\mathbb{C}$.
Define
\[
G_{F,u}(\mathbf x,\mathbf y):=\Bigl(\sum_{t\in\mathbb{F}_q}|F(\mathbf x, \mathbf y,t)|^u\Bigr)^{\frac{1}{u}}
\quad\text{for }1\le u<\infty,
\qquad
G_{F,\infty}(\mathbf x,\mathbf y):=\max_{t\in\mathbb{F}_q}|F(\mathbf x, \mathbf y,t)|
\quad\text{for }u=\infty.
\]
Then
\[
\|G_{F,u}\|_{\ell^u(\mathbb{F}_q^{2n})}=\|F\|_{\ell^u(\mathbb{H}_n(\mathbb{F}_q))}
\]
and
\[
\mathcal{M}_{\mathbb{H}_n}F([\mathbf{v}])\le \mathcal{M}_{2n}G_{F,u}([\mathbf{v}])
\quad\text{for all }[\mathbf{v}]\in \mathbb{P}^{2n-1}(\mathbb{F}_q),
\]
where $\mathcal{M}_{2n}$ denotes the standard Kakeya maximal operator on $\mathbb{F}_q^{2n}$.
\end{lemma}

\begin{proof}
The identity of norms follows by expanding both sides.
For the domination, fix $[\mathbf{v}]$ and choose a representative $\mathbf{v}=(\mathbf{a}, \mathbf{b})\ne 0$.
Let $\mathbf{p}=(\mathbf x_0, \mathbf y_0,t_0)\in\mathbb{H}_n(\mathbb{F}_q)$ and set $L=L_{\mathbf{p}, [\mathbf{v}]}$.
The projection $\pi(\mathbf x, \mathbf y,t)=(\mathbf x, \mathbf y)$ maps $L$ bijectively to the affine line
\[
\ell=\{(\mathbf x_0, \mathbf y_0)+s(\mathbf a, \mathbf b): s\in\mathbb{F}_q\}\subset \mathbb{F}_q^{2n}.
\]
Therefore,
\[
\sum_{\mathbf x\in L}|F(\mathbf x)|
=
\sum_{(\mathbf x, \mathbf y)\in \ell}|F(\mathbf x, \mathbf y,t(\mathbf x, \mathbf y))|
\le
\sum_{(\mathbf x, \mathbf y)\in \ell} G_{F,u}(\mathbf x, \mathbf y),
\]
where $t(\mathbf x, \mathbf y)$ is the unique $t$ such that $(\mathbf x, \mathbf y,t)\in L$.
Taking the maximum over $\bfp$ gives $\mathcal{M}_{\mathbb{H}_n}F([\mathbf{v}])\le \mathcal{M}_{2n}G_{F,u}([\mathbf{v}])$.
\end{proof}
With the domination lemma in hand, it remains to bound $\mathcal{M}_{2n}G_{F,u}$ in 
$\ell^u(\mathbb{P}^{2n-1}(\mathbb{F}_q))$. This requires a sharp diagonal estimate 
for the standard Kakeya maximal operator on $\mathbb{F}_q^{2n}$, which is provided 
by the following result of Ellenberg, Oberlin, and Tao in \cite{EOT}.
\begin{theorem}[Theorem 1.3, \cite{EOT}]\label{thm:EOT_lines}
Let $d\ge 1$ and let $f:\mathbb{F}_q^d\to\mathbb{R}$.
Let $\mathcal{M}_{d}$ be the Kakeya maximal operator from $\mathbb{F}_q^d$ into $\mathbb{R}$ defined by
\[
\mathcal{M}_{d} f(\omega)
=
\max_{\ell\parallel\omega}\ \sum_{\mathbf z\in \ell} |f(\mathbf z)|
\qquad
(\omega\in \mathbb{P}^{d-1}(\mathbb{F}_q)),
\]
where the maximum is over all affine lines $\ell\subset\mathbb{F}_q^d$ with direction $\omega$.
Then
\[
\|\mathcal{M}_d f\|_{\ell^d(\mathbb{P}^{d-1}(\mathbb{F}_q))}
\le
C_d\, q^{\frac{d-1}{d}}\,\|f\|_{\ell^d(\mathbb{F}_q^d)},
\]
where $C_d$ depends only on $d$.
\end{theorem}

\begin{lemma}\label{lem:Hn_diagonal}
Let $n\ge 1$.
For every $1\le u\le \infty$ and every $F:\mathbb{H}_n(\mathbb{F}_q)\to\mathbb{C}$ one has
\[
\|\mathcal{M}_{\mathbb{H}_n}F\|_{\ell^u(\mathbb{P}^{2n-1}(\mathbb{F}_q))}
\le
C_{n,u}\, q^{\tau_{2n}(u)}\, \|F\|_{\ell^u(\mathbb{H}_n(\mathbb{F}_q))},
\]
where
\[
\tau_{2n}(u)
=
\begin{cases}
\frac{2n-1}{u}, & 1\le u\le 2n,\\[4pt]
1-\frac1u, & 2n\le u\le \infty,
\end{cases}
\]
and
\[
C_{n,u}:=
\begin{cases}
2^{\,1-\theta(u)}\,C_{2n}^{\,\theta(u)}, & 1\le u\le 2n,\\[0.4em]
C_{2n}^{\,\frac{2n}{u}}, & 2n\le u\le \infty,
\end{cases}
\qquad
\theta(u):=\frac{2n}{2n-1}\Bigl(1-\frac1u\Bigr).
\]
\end{lemma}

\begin{proof}
{ 
Let $G=G_{F,u}$ be the non-negative valued function on $\mathbb{F}_q^{2n}$ associated with $F$ as in Lemma~\ref{lem:Hn_domination}. By that lemma,
\[
\|\mathcal{M}_{\mathbb{H}_n}F\|_{\ell^u(\mathbb{P}^{2n-1}(\mathbb{F}_q))}
\le
\|\mathcal{M}_{2n}G\|_{\ell^u(\mathbb{P}^{2n-1}(\mathbb{F}_q))},
\qquad
\|G\|_{\ell^u(\mathbb{F}_q^{2n})}
=
\|F\|_{\ell^u(\mathbb{H}_n(\mathbb{F}_q))}.
\]
We prove the corresponding estimate for $\mathcal{M}_{2n}G$ by linearization.
For each direction $[\mathbf v]\in\mathbb P^{2n-1}(\mathbb F_q)$, choose an affine line
$\ell_{[\mathbf v]}\subset\mathbb F_q^{2n}$ with direction $[\mathbf v]$ such that
\[
\mathcal{M}_{2n}G([\mathbf v])
=
\sum_{\mathbf x\in\ell_{[\mathbf v]}}G(\mathbf x),
\]
which is possible since the field is finite and $G\ge 0$. Define the linear operator
\[
(Tf)([\mathbf v])
:=
\sum_{\mathbf x\in\ell_{[\mathbf v]}} f(\mathbf x),
\qquad [\mathbf v]\in\mathbb P^{2n-1}(\mathbb F_q).
\]
Then $TG=\mathcal{M}_{2n}G$ pointwise. Moreover, for every $f$,
\[
|(Tf)([\mathbf v])|
\le
\sum_{\mathbf x\in\ell_{[\mathbf v]}} |f(\mathbf x)|
\le
\mathcal{M}_{2n}f([\mathbf v]),
\]
where the maximal operator on the right is applied to $|f|$.
Hence Theorem~\ref{thm:EOT_lines} gives the uniform endpoint bound
\[
\|Tf\|_{\ell^{2n}(\mathbb{P}^{2n-1}(\mathbb{F}_q))}
\le
C_{2n}\,q^{\frac{2n-1}{2n}}
\|f\|_{\ell^{2n}(\mathbb{F}_q^{2n})}.
\]
We also have the trivial $\ell^1$ endpoint
\[
\begin{aligned}
\|Tf\|_{\ell^1(\mathbb{P}^{2n-1}(\mathbb{F}_q))}
&\le
\sum_{[\mathbf v]\in\mathbb P^{2n-1}(\mathbb F_q)}
\sum_{\mathbf x\in\ell_{[\mathbf v]}} |f(\mathbf x)|  \\
&\le
\bigl|\mathbb P^{2n-1}(\mathbb F_q)\bigr|\,
\|f\|_{\ell^1(\mathbb F_q^{2n})}
\le
2q^{2n-1}\|f\|_{\ell^1(\mathbb F_q^{2n})}.
\end{aligned}
\]
Interpolating these two linear estimates gives, for $1\le u\le 2n$,
\[
\|Tf\|_{\ell^u(\mathbb{P}^{2n-1}(\mathbb{F}_q))}
\le
C_{n,u}\,q^{\frac{2n-1}{u}}
\|f\|_{\ell^u(\mathbb F_q^{2n})},
\]
where
\[
C_{n,u}=2^{\,1-\theta(u)}C_{2n}^{\,\theta(u)},
\qquad
\theta(u):=\frac{2n}{2n-1}\Bigl(1-\frac1u\Bigr).
\]
Applying this to $f=G$ proves the claimed estimate in the range $1\le u\le 2n$.

For $2n\le u\le\infty$, we use the same linearization. The trivial endpoint
\[
\|Tf\|_{\ell^\infty(\mathbb P^{2n-1}(\mathbb F_q))}
\le
q\|f\|_{\ell^\infty(\mathbb F_q^{2n})}
\]
and the $u=2n$ endpoint above interpolate to
\[
\|Tf\|_{\ell^u(\mathbb{P}^{2n-1}(\mathbb{F}_q))}
\le
C_{n,u}\,q^{1-\frac1u}
\|f\|_{\ell^u(\mathbb F_q^{2n})},
\]
where $C_{n,u}=C_{2n}^{\frac{2n}{u}}$. Applying this to $f=G$ proves the claimed estimate in the range $2n\le u\le\infty$.
Combining the two ranges gives
\[
\|\mathcal{M}_{\mathbb{H}_n}F\|_{\ell^u(\mathbb{P}^{2n-1}(\mathbb{F}_q))}
\le
C_{n,u}\,q^{\tau_{2n}(u)}
\|F\|_{\ell^u(\mathbb{H}_n(\mathbb{F}_q))}.
\]
}
\end{proof}

The remaining arguments mirror those in Section~\ref{section4}, with $|\mathbb{P}^{1}(\mathbb{F}_q)|=q+1$ replaced by
$|\mathbb{P}^{2n-1}(\mathbb{F}_q)|$.

\subsection{The case $A_n(u,v)\le \frac{2n-1}{v}
$}
\begin{lemma}\label{lem:Hn_upper_1}
Let $1\le v\le u\le 2n$.
Then
\[
A_n(u,v)\le \frac{2n-1}{v}.
\]
\end{lemma}

\begin{proof}
Since $\mathbb{P}^{2n-1}(\mathbb{F}_q)$ is a finite set, the standard $\ell^p$--embedding implies
\begin{align}\label{embedding-ell^u}
\|\mathcal{M}_{\mathbb{H}_n}F\|_{\ell^v(\mathbb{P}^{2n-1}(\mathbb{F}_q))}
\le
\bigl|\mathbb{P}^{2n-1}(\mathbb{F}_q)\bigr|^{\frac{1}{v}-\frac{1}{u}}\,
\|\mathcal{M}_{\mathbb{H}_n}F\|_{\ell^u(\mathbb{P}^{2n-1}(\mathbb{F}_q))}.
\end{align}
Since $u\le 2n$, Lemma~\ref{lem:Hn_diagonal} gives
\[
\|\mathcal{M}_{\mathbb{H}_n}F\|_{\ell^u(\mathbb{P}^{2n-1}(\mathbb{F}_q))}
\le
C_{n,u}\, q^{\frac{2n-1}{u}}\,\|F\|_{\ell^u(\mathbb{H}_n(\mathbb{F}_q))}.
\]
Combining these estimates with \eqref{eq:Psize} completes the proof.
\end{proof}

\begin{lemma}\label{lem:Hn_upper_4}
Let $1\leq u<2n$ and $u\le v\le (2n-1)\frac{u}{u-1}$ (with the convention $\frac{u}{u-1}=\infty$ when $u=1$).
Then
\[
A_n(u,v)\le \frac{2n-1}{v}.
\]
\end{lemma}

\begin{proof}
When $u=1$ we have 
\[
\|\mathcal{M}_{\mathbb{H}_n}F\|_{\ell^v(\mathbb{P}^{2n-1}(\mathbb{F}_q))}\leq |\mathbb P^{2n-1}(\F_q)|^{\frac{1}{v}}\|F\|_{\ell^1(\mathbb H_n(\F_q))}\lesssim q^{\frac{2n-1}{v}}\|F\|_{\ell^1(\mathbb H_n(\F_q))}.
\]
In the following, we assume that $u>1$. 

Let $\widetilde{F}:=|F|$, the absolute value of $F$. Fix, for each $[\mathbf{v}]\in \mathbb{P}^{2n-1}(\mathbb{F}_q)$, a horizontal line $L_{[\bfv]}$ such that
\[
\mathcal{M}_{\mathbb{H}_n} \widetilde{F}([\mathbf{v}])=\sum_{\mathbf{x}\in L_{[\mathbf{v}]}} \widetilde{F}(\mathbf{x}).
\]
Define the linear operator $T$ on functions $\mathcal G:\mathbb{H}_n(\mathbb{F}_q)\to\mathbb{C}$ by
\[
(T \mathcal G)([\mathbf{v}])=\sum_{\mathbf x\in L_{[\mathbf{v}]}} \mathcal G(\mathbf x).
\]
Then $\|\mathcal{M}_{\mathbb{H}_n}F\|_{\ell^v(\mathbb{P}^{2n-1}(\mathbb{F}_q))}=\|T \widetilde{F}\|_{\ell^v(\mathbb{P}^{2n-1}(\mathbb{F}_q))}$ and $\|\widetilde{F}\|_{\ell^u(\mathbb{H}_n(\mathbb{F}_q))}=\|F\|_{\ell^u(\mathbb{H}_n(\mathbb{F}_q))}$.

Choose $\theta\in(0,1)$ and $r\in[1,\infty]$ such that
\begin{align}\label{definition-r}
    \frac{1}{u} = (1-\theta)\cdot 1 + \theta \cdot \frac{1}{2n},
\qquad
\frac{1}{v} = (1-\theta)\cdot \frac{1}{r} + \theta \cdot \frac{1}{2n}.
\end{align}
Solving the first identity implies
\[
\theta = \frac{2n}{2n-1}\left(1-\frac{1}{u}\right).
\]
The assumption $u \le v \le (2n-1)\frac{u}{u-1}$ is equivalent to
\[
0 \le \frac{1}{r} \le 1,
\]
and therefore such a choice of $r$ is possible.

We record two endpoint bounds for $T$.
For every $\mathcal G$ and every $[\mathbf{v}]$ one has $|(T\mathcal G)([\mathbf{v}])|\le \|\mathcal G\|_{\ell^1(\mathbb{H}_n(\mathbb{F}_q))}$.
Therefore,
\[
\|T\mathcal G\|_{\ell^r(\mathbb{P}^{2n-1}(\mathbb{F}_q))}
\le
\bigl|\mathbb{P}^{2n-1}(\mathbb{F}_q)\bigr|^{\frac{1}{r}}\, \|\mathcal  G\|_{\ell^1(\mathbb{H}_n(\mathbb{F}_q))}.
\]
Also, $|T \mathcal G|\le \mathcal{M}_{\mathbb{H}_n}(|\mathcal G|)$ pointwise.
Lemma~\ref{lem:Hn_diagonal} at $u=2n$ gives
\[
\|T\mathcal G\|_{\ell^{2n}(\mathbb{P}^{2n-1}(\mathbb{F}_q))}
\le
C_{2n}\, q^{\frac{2n-1}{2n}}\, \|\mathcal G\|_{\ell^{2n}(\mathbb{H}_n(\mathbb{F}_q))}.
\]
Applying Lemma~\ref{interpolation-inequality} to $T$ with these two endpoints and using \eqref{definition-r}, we obtain
\[
\|T\mathcal G\|_{\ell^v(\mathbb{P}^{2n-1}(\mathbb{F}_q))}
\le
C_{n,u,v}\,
q^{(2n-1)\bigl(\frac{1-\theta}{r}+\frac{\theta}{2n}\bigr)}\,
\|\mathcal G\|_{\ell^u(\mathbb{H}_n(\mathbb{F}_q))},
\]
where $C_{n,u,v}:=2^{\frac{1}{v}-\frac{\theta(n, u)}{2n}} C_{2n}^{\theta(n, u)}$ and $\theta(n, u):=\frac{2n}{2n-1} \big(1-\frac{1}{u}\big)$.
Since $\frac{1-\theta}{r}+\frac{\theta}{2n}=\frac{1}{v}$, this becomes
\[
\|T \widetilde{F}\|_{\ell^v(\mathbb{P}^{2n-1}(\mathbb{F}_q))}
\le
C_{n,u,v}\, q^{\frac{2n-1}{v}}\, \|\widetilde{F} \|_{\ell^u(\mathbb{H}_n(\mathbb{F}_q))},
\]
which implies the claim.
\end{proof}

\subsection{The case $A_n(u,v)\le 1-\frac1u$}
\begin{lemma}\label{lem:Hn_upper_2}
Let $2n\le u\le v\le \infty$.
Then
\[
A_n(u,v)\le 1-\frac1u.
\]
\end{lemma}

\begin{proof}
Since $\mathbb{P}^{2n-1}(\mathbb{F}_q)$ is a finite set and $u \le v$, the monotonicity of $\ell^p$ norms implies that
$$ \|\mathcal{M}_{\mathbb{H}_n}F\|_{\ell^v(\mathbb{P}^{2n-1}(\mathbb{F}_q))}\le \|\mathcal{M}_{\mathbb{H}_n}F\|_{\ell^u(\mathbb{P}^{2n-1}(\mathbb{F}_q))}. $$
Moreover, by Lemma~\ref{lem:Hn_diagonal}, we have
$$ \|\mathcal{M}_{\mathbb{H}_n}F\|_{\ell^u(\mathbb{P}^{2n-1}(\mathbb{F}_q))}\le C_{n,u} q^{1-\frac{1}{u}}\|F\|_{\ell^u(\mathbb{H}_n(\mathbb{F}_q))}. 
$$
This proves the claim.
\end{proof}

\begin{lemma}\label{lem:Hn_upper_5}
Let $1\le u\le 2n$ and $v\ge (2n-1)\frac{u}{u-1}$.
Then
\[
A_n(u,v)\le 1-\frac1u.
\]
\end{lemma}

\begin{proof}
When $u=1$, the condition on $v$ forces $v=\infty$, and the bound evaluates exactly to 1, perfectly matching $q^{1-1/u}$. 

Now assume that $1<u\le 2n$.
By the linearization argument in Lemma~\ref{lem:Hn_upper_4}, it suffices to prove the bound for the associated linear operator $T$.
We use the endpoint estimates
\[
\|T \mathcal G\|_{\ell^\infty(\mathbb{P}^{2n-1}(\mathbb{F}_q))}\le \|\mathcal G\|_{\ell^1(\mathbb{H}_n(\mathbb{F}_q))},
\qquad
\|T\mathcal G\|_{\ell^{2n}(\mathbb{P}^{2n-1}(\mathbb{F}_q))}\le C_{2n}\, q^{\frac{2n-1}{2n}}\, \|\mathcal G\|_{\ell^{2n}(\mathbb{H}_n(\mathbb{F}_q))}.
\]
Let $\theta\in (0,1]$ satisfy $\frac1u=(1-\theta)\cdot 1+\theta\cdot\frac1{2n}$.
Then $\theta=\frac{2n}{2n-1}\left(1-\frac1u\right)$.
Interpolation implies
\[
\|T \mathcal G\|_{\ell^{\frac{2n}{\theta}}(\mathbb{P}^{2n-1}(\mathbb{F}_q))}
\le
C_{n,u}\, q^{\frac{(2n-1)\theta}{2n}}\, \|\mathcal G\|_{\ell^u(\mathbb{H}_n(\mathbb{F}_q))}
=
C_{n,u}\, q^{1-\frac{1}{u}}\,\|\mathcal G\|_{\ell^u(\mathbb{H}_n(\mathbb{F}_q))}.
\]
Since $\frac{2n}{\theta}=(2n-1)\frac{u}{u-1}$ and $v\ge (2n-1)\frac{u}{u-1}$, one has
\[
\|T \mathcal G\|_{\ell^v(\mathbb{P}^{2n-1}(\mathbb{F}_q))}
\le \|T \mathcal G\|_{\ell^{\frac{(2n-1)u}{u-1}}(\mathbb{P}^{2n-1}(\mathbb{F}_q))}
\le C_{n,u}\, q^{1-\frac{1}{u}}\, \|\mathcal G\|_{\ell^u(\mathbb{H}_n(\mathbb{F}_q))}.
\]
This implies the claim.
\end{proof}

\subsection{The case $
A_n(u,v)\le 1+\frac{2n-1}{v}-\frac{2n}{u}$}
\begin{lemma}\label{lem:Hn_upper_3}
Let $2n\le u\le \infty$ and $1\le v\le u$.
Then
\[
A_n(u,v)\le 1+\frac{2n-1}{v}-\frac{2n}{u}.
\]
\end{lemma}

\begin{proof}
It follows from \eqref{embedding-ell^u} and Lemma~\ref{lem:Hn_diagonal} that 
\begin{align*}
\|\mathcal{M}_{\mathbb{H}_n}F\|_{\ell^v(\mathbb{P}^{2n-1}(\mathbb{F}_q))}
\le &
\bigl|\mathbb{P}^{2n-1}(\mathbb{F}_q)\bigr|^{\frac{1}{v}-\frac{1}{u}}\,
\|\mathcal{M}_{\mathbb{H}_n}F\|_{\ell^u(\mathbb{P}^{2n-1}(\mathbb{F}_q))} \\
\leq & 
\bigl|\mathbb{P}^{2n-1}(\mathbb{F}_q)\bigr|^{\frac{1}{v}-\frac{1}{u}} \bigg( C_{n,u} q^{1-\frac{1}{u}}\|F\|_{\ell^u(\mathbb{H}_n(\mathbb{F}_q))}  \bigg).
\end{align*}
Using \eqref{eq:Psize} implies
\[
\|\mathcal{M}_{\mathbb{H}_n}F\|_{\ell^v(\mathbb{P}^{2n-1}(\mathbb{F}_q))}
\le
C_{n,u,v}\, q^{(2n-1)(\frac{1}{v}-\frac{1}{u})}\, q^{1-\frac{1}{u}}\, \|F\|_{\ell^u(\mathbb{H}_n(\mathbb{F}_q))}
=
C_{n,u,v}\, q^{1+\frac{2n-1}{v}-\frac{2n}{u}}\,\|F\|_{\ell^u(\mathbb{H}_n(\mathbb{F}_q))}.
\]
This proves the claim.
\end{proof}

The lower bounds follow from the same three test functions as in Section~\ref{section5}.
We include the proofs for completeness.

\subsection{The case $A_n(u,v)\ge \frac{2n-1}{v}$}
\begin{lemma}\label{lem:Hn_lower_1}
For every $1\le u,v\le \infty$ one has $A_n(u,v)\ge \frac{2n-1}{v}$.
\end{lemma}

\begin{proof}
Fix $\bfp_\ast\in \mathbb{H}_n(\mathbb{F}_q)$ and set $F=\delta_{\bfp_\ast}$.
For every direction $[\mathbf{v}]\in \mathbb{P}^{2n-1}(\mathbb{F}_q)$ there is a unique horizontal line in direction $[\mathbf{v}]$ through $\bfp_\ast$.
Therefore, $\mathcal{M}_{\mathbb{H}_n}F([\mathbf{v}])=1$ for all $[\mathbf{v}]$.
Together with \eqref{eq:Psize}, we obtain
\[
\|\mathcal{M}_{\mathbb{H}_n}F\|_{\ell^v(\mathbb{P}^{2n-1}(\mathbb{F}_q))}
=
\bigl|\mathbb{P}^{2n-1}(\mathbb{F}_q)\bigr|^{\frac{1}{v}} \geq q^{\frac{2n-1}{v}},
\qquad
\|F\|_{\ell^u(\mathbb{H}_n(\mathbb{F}_q))}=1.
\]
By the definition of $A_n(u,v)$, we obtain $A_n(u,v)\ge \frac{2n-1}{v}$.
\end{proof}

\subsection{The case $A_n(u,v)\ge 1-\frac1u$}
\begin{lemma}\label{lem:Hn_lower_2}
For every $1\le u,v\le \infty$ one has $A_n(u,v)\ge 1-\frac1u$.
\end{lemma}

\begin{proof}
Let $L\subset \mathbb{H}_n(\mathbb{F}_q)$ be a fixed horizontal line and set $F=\mathbf{1}_L$.
Then $\|F\|_{\ell^u(\mathbb{H}_n(\mathbb{F}_q))}=|L|^{\frac{1}{u}}=q^{\frac{1}{u}}$.
Let $[\mathbf v_0]$ be the direction of $L$.
For this direction one may choose the basepoint $\bfp$ so that $L=L_{\mathbf p, [\mathbf v_0]}$.
Therefore, $\mathcal{M}_{\mathbb{H}_n}F([\mathbf v_0])\ge \sum\limits_{\mathbf x\in L}\mathbf{1}_L(\mathbf x)=q$, and hence
\[
\|\mathcal{M}_{\mathbb{H}_n}F\|_{\ell^v(\mathbb{P}^{2n-1}(\mathbb{F}_q))}\ge q
\]
 for every $1\le v\le\infty$.
This forces $A_n(u,v)\ge 1-\frac1u$.
\end{proof}

\subsection{The case $A_n(u,v)\ge 1+\frac{2n-1}{v}-\frac{2n}{u}$}
\begin{lemma}\label{lem:Hn_lower_3}
For every $1\le u,v\le \infty$ one has $A_n(u,v)\ge 1+\frac{2n-1}{v}-\frac{2n}{u}$.
\end{lemma}

\begin{proof}
Let
\[
S:=\bigcup_{[\mathbf{v}]\in \mathbb{P}^{2n-1}(\mathbb{F}_q)} L_{\mathbf{0}, [\mathbf{v}]},
\qquad
F:=\mathbf{1}_S.
\]
Distinct horizontal lines through $\bf0$ intersect only at $\bf0$.
Since each such line has $q$ points and there are $|\mathbb{P}^{2n-1}(\mathbb{F}_q)|$ directions, together with \eqref{size-P^{2n-1}}, we obtain
\[
|S|
=
1+\bigl|\mathbb{P}^{2n-1}(\mathbb{F}_q)\bigr|(q-1)
=
q^{2n}.
\]
Therefore, $\|F\|_{\ell^u(\mathbb{H}_n(\mathbb{F}_q))}=|S|^{\frac{1}{u}}=q^{\frac{2n}{u}}$.
For every direction $[\mathbf{v}]$ we have $L_{\bf0, [\mathbf{v}]}\subset S$, and hence $\mathcal{M}_{\mathbb{H}_n}F([\mathbf{v}])\ge q$.
Thus,
\[
\|\mathcal{M}_{\mathbb{H}_n}F\|_{\ell^v(\mathbb{P}^{2n-1}(\mathbb{F}_q))}
\ge
q\, \bigl|\mathbb{P}^{2n-1}(\mathbb{F}_q)\bigr|^{\frac{1}{v}}
\geq
q^{1+\frac{2n-1}{v}}
\]
by \eqref{eq:Psize}.

This forces $A_n(u,v)\ge 1+\frac{2n-1}{v}-\frac{2n}{u}$.
\end{proof}

\subsection{Proof of Theorem~\ref{thm:exact_Hn}}
\begin{proof}[Proof of Theorem~\ref{thm:exact_Hn}]
Lemmas~\ref{lem:Hn_lower_1}, \ref{lem:Hn_lower_2}, and \ref{lem:Hn_lower_3} give
\[
A_n(u,v)\ge \max\Bigg\{\frac{2n-1}{v},\, 1-\frac1u,\, 1+\frac{2n-1}{v}-\frac{2n}{u}\Bigg\}.
\]

For the matching upper bound,  Lemmas~\ref{lem:Hn_upper_1}-\ref{lem:Hn_upper_3} cover all the pairs $(u,v)$. Therefore,
\[
A_n(u,v)\le \max\Bigg\{\frac{2n-1}{v},\, 1-\frac1u,\, 1+\frac{2n-1}{v}-\frac{2n}{u}\Bigg\},
\]
and the theorem follows.
\end{proof}

\section{Proof of Theorem \ref{thm:sr-l2-bound}}\label{section8}

Recall that $\mathcal L$ is the set of all horizontal lines in $\mathbb H_1(\F_q)$, and $\cL(\omega)$ is the set of all lines in $\cL$ with refined direction $\omega$. 

For $[\bfv] \in \PP^1(\Fq)$, $\omega\in \cD_1$, or $\bfp\in \HH_1(\Fq)$, denote by $\cL([\bfv])$, $\cL(\omega)$ and $\cL(\bfp)$ the set of all lines in $\cL$ with direction $[\bfv]$, with refined direction $\omega$, or through point $\bfp$, respectively.

For any $\omega=[a:b:c]\in \cD_1$, we take any fixed $(x_\omega,y_\omega)\in \Fq^2$ with $x_\omega b-y_\omega a=c$. One checks that 
\begin{equation} \label{eq_q_lines}
\cL(\omega)=\Big\{L_{(x_\omega,y_\omega,t),[a:b]}=\{(x_\omega+sa,y_\omega+sb,t+sc):\, s\in\Fq\}:\,  t\in \Fq\Big\}.
\end{equation}
Let $\pi:\mathbb H_1(\mathbb F_q)\to\mathbb F_q^2$ be the projection defined by $\pi(x,y,t)=(x,y)$. Let $\ell_\omega$ be the affine line on $\Fq^2$ given by
\[
\ell_\omega=\{(x_\omega+sa,y_\omega+sb):\, s\in \Fq\}=\{(x,y)\in \F_q^2:\, bx-ay=c\}.
\]
Then $\pi(L)=\ell_\omega$ for any $L\in \cL(\omega)$. Indeed, the map $\omega\mapsto \ell_\omega$ is bijective from $\cD_1$ to the set of all affine lines in $\Fq^2$. Moreover, the line $\ell_{[a:b:c]}$, with $[a:b:c]\in \cD_1$, is parallel to $[a:b]$. For each $[\bfv]\in \mathbb P^1(\Fq)$, there are exactly $q$ elements $\omega$ of $\cD_1$ such that $\ell_\omega$ is parallel to $[\bfv]$.

For a given choice of the family $\{L_\omega\}_{\omega\in \cD_1}$ with $L_\omega\in \cL(\omega)$, we will use the explicit parametrizations as follows. Consider the disjoint union $\cD_1=\cD_1'\cup \cD_1''$, with 
\[
\cD_1'=\{[1:m:\gamma]:\, m,\gamma \in \Fq\},\quad \cD_1''=\{[0:1:\gamma]:\, \gamma \in \Fq\}.
\]
{ 
For each $\omega=[1:m:\gamma]\in \cD_1^{\prime}$ and each $\tau\in \F_q$, set
\begin{equation} \label{eq_q_lines_1}
L_{\omega,\tau}:=\{(x,mx-\gamma,\tau+\gamma x):\, x\in\Fq\}. 
\end{equation}
Then $\cL(\omega)=\{L_{\omega,\tau}: \tau \in \F_q\}$, so any chosen line $L_\omega \in \cL(\omega)$ takes the form $L_{\omega, \tau_\omega}$ for a unique $\tau_\omega \in \F_q$.

For each $\omega=[0:1:\gamma]\in \cD_1^{\prime\prime}$ and each $\tau\in\F_q$, set
\begin{equation} \label{eq_q_lines_2}
L_{\omega,\tau}:=\{(\gamma,y,\tau+\gamma y):\, y\in\Fq\}. 
\end{equation}
Then $\cL(\omega)=\{L_{\omega,\tau}: \tau\in\F_q\}$, so any chosen line $L_\omega\in\cL(\omega)$ also takes the form $L_{\omega,\tau_\omega}$ for a unique $\tau_\omega\in\F_q$.
}

\begin{proof}[Proof of Theorem~\ref{thm:sr-l2-bound}]

It suffices to prove the claim for non-negative $F$. For each $\omega\in\mathcal D_1$, we choose an $L_\omega\in \cL(\omega)$ such that 
\[
(\mathcal{M}^{\mathrm{rd}}_{\mathbb H_1}F)(\omega)
=
\sum_{\mathbf p\in L_\omega}F(\mathbf p).
\]
For this family $\{L_\omega\}_{\omega\in \cD_1}$, we will use the parametrizations in \eqref{eq_q_lines_1} and \eqref{eq_q_lines_2} later. 

Let $T:\, \ell^2(\mathbb H_1(\Fq))\rightarrow \ell^2(\cD_1)$ be the linear transformation given by 
\[
(Tf)(\omega):=\sum_{\mathbf p\in L_\omega} f(\mathbf p),
\qquad \omega\in\mathcal D_1.
\]
Theorem \ref{thm:sr-l2-bound} will follow if we can prove 
\begin{equation} \label{eq_aim_thm1.7}
\|Tf\|_{\ell^2(\mathcal D_1)}
\leq \, 4q^{\frac{1}{2}}\,\|f\|_{\ell^2(\mathbb H_1(\mathbb F_q))}
\end{equation}
for any $f\in \ell^2(\mathbb H_1(\Fq))$. 

For $(x,y)\in \Fq^2$, we define the Fourier transform in the central variable $t$ by
\[
\widehat f(x,y;\xi):=\sum_{t\in\mathbb F_q} f(x,y,t)\,\chi(-\xi t),
\qquad \xi\in\mathbb F_q,
\]
where $\chi$ is a fixed non-trivial additive character. The Fourier inversion gives
\[
f(x,y,t)=\frac1q\sum_{\xi\in\mathbb F_q}\widehat f(x,y;\xi)\,\chi(\xi t),
\]
Applying the Plancherel theorem in $t$ and summing over all $(x,y)$ implies
\begin{equation} \label{eq_Plancerel_central}
\sum_{x,y\in\mathbb F_q}\sum_{\xi\in\mathbb F_q}
\bigl|\widehat f(x,y;\xi)\bigr|^2
=
q\,\|f\|_{\ell^2(\mathbb H_1(\mathbb F_q))}^2.
\end{equation}

Now for any $\omega\in \cD_1$, 
\begin{align*}
(Tf)(\omega)
=\sum_{(x,y,t)\in L_\omega} f(x,y,t)=\frac1q\sum_{\xi\in\mathbb F_q}\sum_{(x,y,t)\in L_\omega}
\widehat f(x,y;\xi)\,\chi(\xi t)=\sum\limits_{\xi\in \Fq} (T_\xi f)(\omega),
\end{align*}
where $T_\xi$ are linear transformations given by 
\begin{equation} \label{eq_T_xi}
(T_\xi f)(\omega)
:=
\frac1q\sum_{(x,y,t)\in L_\omega}
\widehat f(x,y;\xi)\,\chi(\xi t).
\end{equation}

In the following, we consider 
\[
Tf=T_0 f+T_{\neq 0} f, \quad T_{\neq 0} f=\sum_{\xi\in\mathbb F_q^\ast }T_\xi f.
\]
Our aim is to prove 
\begin{equation} \label{eq_conclusion_T0}
\|T_0 f\|_{\ell^2(\mathcal D_1)}
\leq \sqrt{2} q^{\frac12}\,\|f\|_{\ell^2(\mathbb H_1(\mathbb F_q))},
\end{equation}
and 
\begin{equation} \label{eq_conclusion_Tnot0}
\|T_{\neq 0}f\|_{\ell^2(\mathcal D_1)}
\leq
\sqrt{3} q^{\frac{1}{2}}\,\|f\|_{\ell^2(\mathbb H_1(\mathbb F_q))}.
\end{equation}
Then 
\[
\|Tf\|_{\ell^2(\mathcal D_1)}
\le
\|T_0 f\|_{\ell^2(\mathcal D_1)}+\|T_{\neq 0}f\|_{\ell^2(\mathcal D_1)}
\leq (\sqrt{2}+\sqrt{3}) \,q^{\frac{1}{2}}\,\|f\|_{\ell^2(\mathbb H_1(\mathbb F_q))} \leq 4 \,q^{\frac{1}{2}}\,\|f\|_{\ell^2(\mathbb H_1(\mathbb F_q))}
\]
and \eqref{eq_aim_thm1.7} holds.

\noindent\textbf{Case 1: $\xi=0$.} Write $g_0:\, \Fq^2\rightarrow \mathbb{C}$ by 
\begin{equation} \label{eq_def_g0}
g_0(x,y):=\widehat f(x,y;0)=\sum_{t\in\mathbb F_q} f(x,y,t).
\end{equation}
Then, with $\pi(L_{\omega})=\ell_\omega$, one has
\begin{equation} \label{eq_intog}
(T_0 f)(\omega)
=\frac1q\sum_{(x,y,t)\in L_\omega}\widehat f(x,y;0)
=\frac1q\sum_{(x,y)\in \ell_\omega} g_0(x,y).
\end{equation}

For each slope $[\mathbf v]\in \mathbb P^1(\mathbb F_q)$, let $\mathcal L([\mathbf v])$ be the family of all affine
lines in $\mathbb F_q^2$ with direction $[\mathbf v]$. Then $|\mathcal L([\mathbf v])|=q$, and these families
partition the set of all affine lines in $\mathbb F_q^2$. Moreover, by definition \eqref{eq:Mpl}, for every
$\ell\in\mathcal L([\mathbf v])$ one has
\begin{equation} \label{eq_leqM2g}
\sum_{(x,y)\in \ell} |g_0(x,y)| \le \mathcal M_2(|g_0|)([\mathbf v]).
\end{equation}
Therefore, grouping the $\ell^2(\mathcal D_1)$-sum by slopes, together with inserting \eqref{eq_intog} and \eqref{eq_leqM2g}, gives
\begin{align*}
\|T_0 f\|_{\ell^2(\mathcal D_1)}^2& \leq \sum_{\omega\in\mathcal D_1}\Bigl(\frac{1}{q}\sum_{(x,y)\in \ell_\omega} |g_0(x,y)|\Bigr)^2
\le \frac{1}{q^2}\sum_{[\mathbf v]\in\mathbb P^1(\mathbb F_q)}
\sum_{\ell\in\mathcal L([\mathbf v])}\Bigl(\sum_{(x,y)\in \ell} |g_0(x,y)|\Bigr)^2 \nonumber\\
&\le \frac{1}{q^2}\sum_{[\mathbf v]\in\mathbb P^1(\mathbb F_q)} q\cdot
\bigl(\mathcal M_2(|g_0|)([\mathbf v])\bigr)^2 = \frac{1}{q}\,\|\mathcal M_2(|g_0|)\|_{\ell^2(\mathbb P^1(\mathbb F_q))}^2. 
\end{align*}
Using Lemma \ref{lem:planar-l2}, \eqref{eq_def_g0} and the Cauchy-Schwarz inequality sequentially, we obtain that
\begin{align*}
\|\mathcal M_2(|g_0|)\|_{\ell^2(\mathbb P^1(\mathbb F_q))}^2 &\leq 2q \|g_0\|_{\ell^2(\mathbb F_q^2)}^2= 2q\sum\limits_{(x,y)}\Bigl|\sum\limits_t f(x,y,t)\Bigr|^2\\
&\le 2q\, \sum\limits_{(x,y)} q\sum_t |f(x,y,t)|^2=2q^2\,\|f\|_{\ell^2(\mathbb H_1(\mathbb F_q))}^2.
\end{align*}
Therefore \eqref{eq_conclusion_T0} holds.

\noindent\textbf{Case 2: $\xi\neq 0$.} 
When $\xi\in\mathbb F_q^*$, define, for $\omega=[1:m:\gamma]$,
\[
U_\xi(m,\gamma)
:=
\sum_{x\in\mathbb F_q}
\widehat f(x,mx-\gamma;\xi)\,\chi(\xi\gamma x),
\]
and for $\omega=[0:1:\gamma]$,
\[
U_\xi^\infty(\gamma)
:=
\sum_{y\in\mathbb F_q}
\widehat f(\gamma,y;\xi)\,\chi(\xi\gamma y).
\]

We assume at the moment that for every $\xi\in\mathbb F_q^\ast$,
\begin{equation}\label{eq:key-nonvertical}
\sum_{m\in\mathbb F_q}\sum_{\gamma\in\mathbb F_q}
|U_\xi(m,\gamma)|^2
\leq\,2q
\sum_{x,y\in\mathbb F_q}|\widehat f(x,y;\xi)|^2,
\end{equation}
and
\begin{equation}\label{eq:key-vertical}
\sum_{\gamma\in\mathbb F_q}
|U_\xi^\infty(\gamma)|^2
\le
q
\sum_{x,y\in\mathbb F_q}|\widehat f(x,y;\xi)|^2.
\end{equation}
Inserting the explicit parametrizations in \eqref{eq_q_lines_1} and \eqref{eq_q_lines_2} into \eqref{eq_T_xi}, we obtain
\[
(T_\xi f)(\omega)
=
\frac1q\,\chi(\xi\tau_\omega)\,
\sum_{x\in\mathbb F_q}\widehat f(x,mx-\gamma;\xi)\,\chi(\xi\gamma x)
=\frac1q\,\chi(\xi\tau_\omega)\,U_\xi(m,\gamma),\quad \omega=[1:m:\gamma],
\]
\[
(T_\xi f)(\omega)
=
\frac1q\,\chi(\xi\tau_\omega)\,
\sum_{y\in\mathbb F_q}\widehat f(\gamma,y;\xi)\,\chi(\xi\gamma y)
=\frac1q\,\chi(\xi\tau_\omega)\,U_\xi^\infty(\gamma),\quad \omega=[0:1:\gamma].
\]
Since $|\chi(\xi\tau_\omega)|=1$, we obtain from \eqref{eq:key-nonvertical}--\eqref{eq:key-vertical} that
\begin{equation} \label{eq_l2upperbound}
\|T_\xi f\|_{\ell^2(\mathcal D_1)}^2
=
\sum_{m,\gamma}\Bigl|\frac1q U_\xi(m,\gamma)\Bigr|^2
+
\sum_{\gamma}\Bigl|\frac1q U_\xi^\infty(\gamma)\Bigr|^2
\leq { \frac{3}{q}} \sum_{x,y\in\mathbb F_q}|\widehat f(x,y;\xi)|^2.
\end{equation}
It follows by triangle inequality, Cauchy--Schwarz inequality, \eqref{eq_l2upperbound} and  \eqref{eq_Plancerel_central} that
\begin{align*}
\Bigl\|\sum_{\xi\in\mathbb F_q^*}T_\xi f\Bigr\|_{\ell^2(\mathcal D_1)}^2
&\le
\Big(\sum_{\xi\in\mathbb F_q^*}\|T_\xi f\|_{\ell^2(\mathcal D_1)}\Big)^2
\le
q  \sum_{\xi\in\mathbb F_q^*}\|T_\xi f\|_{\ell^2(\mathcal D_1)}^2\\
& \leq { 3}\sum_{\xi\in\mathbb F_q^*}\sum_{x,y\in\mathbb F_q}|\widehat f(x,y;\xi)|^2 \leq 
{ 3}q\,\|f\|_{\ell^2(\mathbb H_1(\mathbb F_q))}^2.
\end{align*}
Then \eqref{eq_conclusion_Tnot0} follows, and Theorem \ref{thm:sr-l2-bound} holds.

\medskip

In the rest of the proof, we shall prove \eqref{eq:key-nonvertical} and \eqref{eq:key-vertical}. Fix $\xi\neq 0$, and write
\[
h_m(\gamma):=U_\xi(m,\gamma)=\sum_{x\in\mathbb F_q}
\widehat f(x,mx-\gamma;\xi)\,\chi(\xi\gamma x).
\]
With the standard Fourier transform in $\gamma$, i.e., 
\[
\widehat h_m(\rho)
=
\sum_{\gamma} h_m(\gamma)\chi(-\rho \gamma), \qquad \rho\in \Fq,
\]
the Plancherel theorem gives
\begin{equation} \label{eq_standard_Plan}
\sum_{\gamma}|h_m(\gamma)|^2
=
\frac1q\sum_{\rho\in\mathbb F_q}|\widehat h_m(\rho)|^2.
\end{equation}
By making the change of variables $y=mx-\gamma$, we obtain
\[
\widehat h_m(\rho)
=
\sum_{\gamma,x}
\widehat f(x,mx-\gamma;\xi)\,
\chi\!\left((\xi x-\rho)\gamma \right)
=
\sum_{x,y}
\widehat f(x,y;\xi)\,
\chi\!\left((\xi x-\rho)(mx-y)\right).
\]
Define
\[
G_\rho(x):=\sum_{y}\widehat f(x,y;\xi)\,
\chi\!\left(-(\xi x-\rho)y\right),
\qquad
Q_\rho(x):=(\xi x-\rho)x.
\]
Then
\[
\widehat h_m(\rho)
=
\sum_{x} G_\rho(x)\,\chi\!\bigl(mQ_\rho(x)\bigr).
\]
Summing in $m$, we have
\begin{align}
\sum_{m}|\widehat h_m(\rho)|^2
&= \sum\limits_{x,x'}G_\rho(x)\overline{G_\rho(x')}\sum\limits_{m} \chi(m(Q_\rho(x)-Q_\rho(x'))) \nonumber\\
&= q\sum\limits_{x,x' \atop Q_\rho(x)=Q_\rho(x')}G_\rho(x)\overline{G_\rho(x')} =
q\sum_{t\in\mathbb F_q}
\Bigl|\sum_{x:\, Q_\rho(x)=t} G_\rho(x)\Bigr|^2. \nonumber 
\end{align}
Since $\xi\neq 0$, for any given $t$ the equation $Q_\rho(x)=t$ has at most $2$ solutions. 
Therefore,
\begin{equation} \label{eq_restrict}
\sum_{m}|\widehat h_m(\rho)|^2
\le
2q\sum_{x}|G_\rho(x)|^2.
\end{equation}
Moreover, for any given $x$, we have 
\begin{equation} \label{eq_G_v2}
\sum\limits_{\rho}|G_\rho(x)|^2 = \sum\limits_{y,y'}\widehat{f}(x,y;\xi)\overline{\widehat{f}(x,y';\xi)}\chi(-\xi x(y-y'))\sum\limits_{\rho}\chi(\rho(y-y')) = q \sum\limits_{y} |\widehat{f}(x,y;\xi)|^2.
\end{equation}
Now, combining \eqref{eq_standard_Plan}-\eqref{eq_G_v2} implies
\[
\sum_{m,\gamma}|U_\xi(m,\gamma)|^2=\sum_{m,\gamma}|h_m(\gamma)|^2
=
\frac1q\sum_{m, \rho\in\mathbb F_q}|\widehat h_m(\rho)|^2\leq 2 \sum\limits_{\rho,x}|G_\rho(x)|^2 =
2q\sum_{x,y}|\widehat f(x,y;\xi)|^2,
\]
which proves \eqref{eq:key-nonvertical}.

For the bound \eqref{eq:key-vertical}, one sees by the Cauchy-Schwarz inequality that 
\[
|U_\xi^\infty(\gamma)|^2=\Big|\sum\limits_{y\in \F_q}\widehat{f}(\gamma,y;\xi)\chi(\xi\gamma y)\Big|^2\leq q\sum\limits_{y\in \F_q}|\widehat{f}(\gamma,y;\xi)|^2.
\]
Summing in $\gamma$ yields \eqref{eq:key-vertical}. 
\end{proof}

Regarding the sharpness of Theorem~\ref{thm:sr-l2-bound}, let $\mathbf 0:=(0,0,0)\in \mathbb H_1(\mathbb F_q)$ and let $F:=\delta_{\mathbf 0}$. Then, $\|F\|_{\ell^2(\mathbb H_1(\mathbb F_q))}=1$.

For each direction $[\mathbf{v}]=[a:b]\in \mathbb P^1(\mathbb F_q)$, consider the horizontal line
$L_{\mathbf 0,[a:b]}$ through $\mathbf 0$ with direction $[a:b]$.
Since $\mathbf 0\in L_{\mathbf 0,[a:b]}$, we have
\[
\sum_{\bfp\in L_{\mathbf 0,[a:b]}} |F(\bfp)| = 1.
\]
{ 
Moreover, by the direct definition of refined direction, using the base point \(\mathbf 0=(0,0,0)\) gives
\[
\mathrm{Dir}\!\left(L_{\mathbf 0,[a:b]}\right) = [a:b:0]\in \mathcal D_1.
\]
}
Therefore, for $\omega\in\mathcal D_1$,
\[
\mathcal{M}^{\mathrm{rd}}_{\mathbb H_1}F(\omega)
=
\begin{cases}
1, & \text{if }\omega=[a:b:0]\text{ for some }[a:b]\in\mathbb P^1(\mathbb F_q),\\
0, & \text{otherwise}.
\end{cases}
\]
Since the set $\{[a:b:0] : [a:b]\in\mathbb P^1(\mathbb F_q)\}$ has cardinality
$|\mathbb P^1(\mathbb F_q)|=q+1$, it follows that
\[
\|\mathcal{M}^{\mathrm{rd}}_{\mathbb H_1}F\|_{\ell^2(\mathcal D_1)}^2
=
\sum_{\omega\in\mathcal D_1}\bigl|\mathcal{M}^{\mathrm{rd}}_{\mathbb H_1}F(\omega)\bigr|^2
=
q+1,
\]
and hence
\[
\|\mathcal{M}^{\mathrm{rd}}_{\mathbb H_1}F\|_{\ell^2(\mathcal D_1)}
=
(q+1)^{\frac{1}{2}}
\sim q^{\frac{1}{2}}.
\]
Consequently, any inequality of the form
$\|\mathcal{M}^{\mathrm{rd}}_{\mathbb H_1}F\|_{\ell^2(\mathcal D_1)}
\le C q^{\alpha}\|F\|_{\ell^2(\mathbb H_1(\mathbb F_q))}$
forces $\alpha\ge \tfrac12$, proving that the exponent $q^{\frac{1}{2}}$ in
Theorem~\ref{thm:sr-l2-bound} is sharp.

\section{Proof of Theorem \ref{thm:exact-rd}}\label{section9}

For each $\omega\in\mathcal D_1$, fix a horizontal line $L_\omega\in \cL(\omega)$,
and define the associated linear operator
\[
(Tf)(\omega):=\sum_{\mathbf p\in L_\omega} f(\mathbf p),
\qquad \omega\in\mathcal D_1.
\]
By the standard linearization argument used previously, for each given $F$ one can choose $\{L_\omega\}_\omega$ so that
$\mathcal{M}^{\mathrm{rd}}_{\mathbb H_1}F = T(|F|)$ pointwise on $\mathcal D_1$.
Consequently, it suffices to prove operator norm bounds for $T$ that are uniform over all choices of $\{L_\omega\}$.

We begin by recording the basic endpoint bounds for any linearization $T$ of 
$\mathcal{M}^{\mathrm{rd}}_{\mathbb{H}_1}$. These follow directly from the fact 
that each line $L_\omega$ has exactly $q$ points and is contained in 
$\mathbb{H}_1(\mathbb{F}_q)$, together with a counting argument on the number of 
refined directions passing through a given point.
\begin{lemma}\label{lem:rd-endpoints}
Let $T$ be any linearization as above. Then for every $f:\mathbb H_1(\mathbb F_q)\to\mathbb C$ one has
\begin{align}
\|Tf\|_{\ell^\infty(\mathcal D_1)} &\le \|f\|_{\ell^1(\mathbb H_1(\mathbb F_q))},\label{eq:rd-endpoint-1infty}\\
\|Tf\|_{\ell^\infty(\mathcal D_1)} &\le q\,\|f\|_{\ell^\infty(\mathbb H_1(\mathbb F_q))},\label{eq:rd-endpoint-inftyinfty}\\
\|Tf\|_{\ell^1(\mathcal D_1)} &\le (q+1)\,\|f\|_{\ell^1(\mathbb H_1(\mathbb F_q))}.\label{eq:rd-endpoint-11}
\end{align}
\end{lemma}

\begin{proof}
The bounds \eqref{eq:rd-endpoint-1infty} and \eqref{eq:rd-endpoint-inftyinfty} follow from
$L_\omega\subset \mathbb H_1(\mathbb F_q)$ and $|L_\omega|=q$.

For \eqref{eq:rd-endpoint-11}, we may assume $f\ge 0$.
Then
\[
\|Tf\|_{\ell^1(\mathcal D_1)}
=
\sum_{\omega\in\mathcal D_1}\ \sum_{\mathbf p\in L_\omega} f(\mathbf p)
=
\sum_{\mathbf p\in\mathbb H_1(\mathbb F_q)} f(\mathbf p)\,N(\mathbf p),
\]
where $N(\mathbf p):=\#\{\omega\in\mathcal D_1:\ \mathbf p\in L_\omega\}$.
If $\mathbf p=(x,y,t)\in L_\omega$, then $(x,y)\in \pi(L_\omega)=\ell_\omega$.
Distinct refined directions $\omega$ correspond to distinct affine lines $\ell_\omega\subset\mathbb F_q^2$,
and there are exactly $q+1$ affine lines through a given point $(x,y)\in\mathbb F_q^2$.
Hence $N(\mathbf p)\le q+1$, proving \eqref{eq:rd-endpoint-11}.
\end{proof}

From the endpoint bounds in Lemma \ref{lem:rd-endpoints} and the sharp $\ell^2$ 
estimate of Theorem \ref{thm:sr-l2-bound}, we can now derive diagonal bounds for 
$\mathcal{M}^{\mathrm{rd}}_{\mathbb{H}_1}$ across the full range of exponents 
$1 \le u \le \infty$ by interpolation. The exponent $\sigma(u)$ reflects the 
two natural ranges: for $1 \le u \le 2$ one interpolates between the $\ell^1 
\to \ell^1$ and $\ell^2 \to \ell^2$ estimates, while for $2 \le u \le \infty$ 
one interpolates between the $\ell^2 \to \ell^2$ and $\ell^\infty \to \ell^\infty$ 
estimates.
\begin{lemma}\label{lem:rd-diagonal}
For every $1\le u\le\infty$ there exists a constant $C_u$ (independent of $q$) such that for all
$F:\mathbb H_1(\mathbb F_q)\to\mathbb C$,
\begin{equation}\label{eq:rd-diagonal}
\|\mathcal{M}^{\mathrm{rd}}_{\mathbb H_1}F\|_{\ell^u(\mathcal D_1)}
\le
C_u\,q^{\sigma(u)}\,\|F\|_{\ell^u(\mathbb H_1(\mathbb F_q))},
\end{equation}
where
\[
\sigma(u):=
\begin{cases}
\frac{1}{u}, & 1\le u\le 2,\\[0.3em]
1-\frac{1}{u}, & 2\le u\le\infty,
\end{cases}
\] 
and 
\[
C_u:=
\begin{cases}
2^{1-\theta(u)} \cdot  5^{\theta(u)}, & 1\le u\le 2,\\[0.3em]
5^{\frac{2}{u}} , & 2\le u\le\infty,
\end{cases}
\]
where $\theta(u)=2\big(1-\frac{1}{u} \big)$, $1 \leq u \leq 2$.
\end{lemma}

\begin{proof}
By linearization, it suffices to prove \eqref{eq:rd-diagonal} for $T$ uniformly.

If $1\le u\le 2$, interpolate between 
$\|T\|_{\ell^1\to\ell^1}\le q+1 \leq 2q$ from
\eqref{eq:rd-endpoint-11} and $\|T\|_{\ell^2\to\ell^2}\leq 5 q^{\frac{1}{2}}$
from Theorem~\ref{thm:sr-l2-bound}. Choose $\theta\in[0,1]$ so that
\[
\frac{1}{u}=(1-\theta)\cdot 1+\theta\cdot\frac{1}{2},
\qquad\text{equivalently }\ \theta=2\Bigl(1-\frac{1}{u}\Bigr).
\]
The Riesz--Thorin theorem gives 
\begin{align*}
    \|T\|_{\ell^u\to\ell^u}\leq 2^{1-\theta} \cdot 5^{\theta} \, q^{1-\frac{\theta}{2}}= 2^{1-\theta} \cdot 5^{\theta} \, q^{\frac{1}{u}}.
\end{align*}
If $2\le u\le\infty$, interpolate between $\|T\|_{\ell^2\to\ell^2}\leq 5 q^{\frac{1}{2}}$ and
$\|T\|_{\ell^\infty\to\ell^\infty}\le q$ from \eqref{eq:rd-endpoint-inftyinfty}.
Choose $\theta\in[0,1]$ so that
\[
\frac{1}{u}=(1-\theta)\cdot\frac{1}{2}+\theta\cdot 0,
\qquad\text{equivalently }\ \theta=1-\frac{2}{u}.
\]
Then, 
\begin{align*}
    \|T\|_{\ell^u\to\ell^u}\leq 5^{1-\theta} q^{\frac{1-\theta}{2}+\theta}=5^{\frac{2}{u}} \, q^{1-\frac{1}{u}}.
\end{align*}
This proves \eqref{eq:rd-diagonal}.
\end{proof}

\subsection{The region $1\le u\le 2$ -- upper bounds}
\begin{lemma}\label{lem:rd-upper-u-le-2}
Let $1\le u\le 2$ and $C_u$ be given in Lemma \ref{lem:rd-diagonal}.
\begin{enumerate}[label=\emph{(\alph*)},itemsep=0.3em]
\item If $1\le v\le u$, then
\[
\|\mathcal{M}^{\mathrm{rd}}_{\mathbb H_1}F\|_{\ell^v(\mathcal D_1)}
\leq C_u \cdot 2^{\frac{1}{v}-\frac{1}{u}} \,
q^{\frac{2}{v}-\frac{1}{u}}\,
\|F\|_{\ell^u(\mathbb H_1(\mathbb F_q))}.
\]
\item If $u\le v\le \frac{u}{u-1}$ (with the convention $\frac{u}{u-1}=\infty$ when $u=1$), then
\[
\|\mathcal{M}^{\mathrm{rd}}_{\mathbb H_1}F\|_{\ell^v(\mathcal D_1)}
\leq 2^{\frac{1}{v}-1+\frac{1}{u}} \cdot 5^{2(1-\frac{1}{u})} \, 
q^{\frac{1}{v}}\,
\|F\|_{\ell^u(\mathbb H_1(\mathbb F_q))}.
\]
\item If $v\ge \frac{u}{u-1}$, then
\[
\|\mathcal{M}^{\mathrm{rd}}_{\mathbb H_1}F\|_{\ell^v(\mathcal D_1)}
\leq 5^{2(1-\frac{1}{u})} \,
q^{1-\frac{1}{u}}\,
\|F\|_{\ell^u(\mathbb H_1(\mathbb F_q))}.
\]
\end{enumerate}
\end{lemma}

\begin{proof}
By linearization it suffices to prove the bounds for $T$ uniformly.

\smallskip
\noindent\textbf{(a)}
Since $v\le u$, Lemma~\ref{lem:embed} gives
\[
\|Tf\|_{\ell^v(\mathcal D_1)}
\le
|\mathcal D_1|^{\frac{1}{v}-\frac{1}{u}}\,
\|Tf\|_{\ell^u(\mathcal D_1)}
\leq 2^{\frac{1}{v}-\frac{1}{u}}
q^{2(\frac{1}{v}-\frac{1}{u})}\,\|Tf\|_{\ell^u(\mathcal D_1)}.
\]
Apply Lemma~\ref{lem:rd-diagonal} with $\sigma(u)=\frac{1}{u}$ to obtain
\[
\|Tf\|_{\ell^v(\mathcal D_1)}
\leq C_u \cdot  2^{\frac{1}{v}-\frac{1}{u}} \, 
q^{2(\frac{1}{v}-\frac{1}{u})}\,q^{\frac{1}{u}}\,\|f\|_{\ell^u(\mathbb H_1(\mathbb F_q))}
=C_u \cdot  2^{\frac{1}{v}-\frac{1}{u}} \,
q^{\frac{2}{v}-\frac{1}{u}}\,\|f\|_{\ell^u(\mathbb H_1(\mathbb F_q))}.
\]

\smallskip
\noindent\textbf{(b)}
Interpolating \eqref{eq:rd-endpoint-11} and \eqref{eq:rd-endpoint-1infty} implies, for every $1\le r\le\infty$,
\begin{equation}\label{eq:rd-1-to-r}
\|Tf\|_{\ell^r(\mathcal D_1)}\le (q+1)^{\frac{1}{r}}\,\|f\|_{\ell^1(\mathbb H_1(\mathbb F_q))}\leq 2^{\frac{1}{r}}  q^{\frac{1}{r}}\,\|f\|_{\ell^1(\mathbb H_1(\mathbb F_q))}.
\end{equation}
Assume first that $u\in(1,2]$.
Choose $\theta\in[0,1]$ so that
\[
\frac{1}{u}=(1-\theta)\cdot 1+\theta\cdot\frac{1}{2},
\qquad\text{so }\ \theta=2\Bigl(1-\frac{1}{u}\Bigr).
\]
The condition $u\le v\le \frac{u}{u-1}$ is equivalent to
\[
\frac{1}{u}\ \ge\ \frac{1}{v}\ \ge\ 1-\frac{1}{u}\ =\ \frac{\theta}{2}.
\]
Hence we may choose $r\in[1,\infty]$ such that
\[
\frac{1}{v}=(1-\theta)\cdot\frac{1}{r}+\theta\cdot\frac{1}{2}.
\]
Interpolating between \eqref{eq:rd-1-to-r} and the $\ell^2\to\ell^2$ estimate from
Theorem~\ref{thm:sr-l2-bound} gives
\begin{align*}
    \|Tf\|_{\ell^v(\mathcal D_1)} \leq & 
(2^{\frac{1}{r}} q^{\frac{1}{r}})^{1-\theta}\,(5 q^{\frac{1}{2}})^\theta\,\|f\|_{\ell^u(\mathbb H_1(\mathbb F_q))}
=2^{\frac{1-\theta}{r}} \cdot 5^{\theta} \,
q^{\frac{1-\theta}{r}+\frac{\theta}{2}}\,\|f\|_{\ell^u(\mathbb H_1(\mathbb F_q))} \\
=& 2^{\frac{1}{v}-1+\frac{1}{u}} \cdot 5^{2(1-\frac{1}{u})} \,
q^{\frac{1}{v}}\,\|f\|_{\ell^u(\mathbb H_1(\mathbb F_q))}.
\end{align*}
The case $u=1$ follows directly from \eqref{eq:rd-1-to-r} with $r=v$.

\smallskip
\noindent\textbf{(c)}
Set $v_0:=\frac{u}{u-1}$, so that $\frac{1}{v_0}=1-\frac{1}{u}$.
Interpolating between $\ell^1\to\ell^\infty$ (norm $\le 1$, \eqref{eq:rd-endpoint-1infty})
and $\ell^2\to\ell^2$ (norm $\leq 5 q^{\frac{1}{2}}$, Theorem~\ref{thm:sr-l2-bound}) with the same
$\theta=2(1-\frac{1}{u})$ implies
\[
\|Tf\|_{\ell^{v_0}(\mathcal D_1)}\leq 5^{\theta} q^{\frac{\theta}{2}}\,\|f\|_{\ell^u(\mathbb H_1(\mathbb F_q))}=5^{2(1-\frac{1}{u})} \, q^{1-\frac{1}{u}}\,\|f\|_{\ell^u(\mathbb H_1(\mathbb F_q))}.
\]
If $v\ge v_0$, then Lemma~\ref{lem:embed} gives
$\|Tf\|_{\ell^v(\mathcal D_1)}\le \|Tf\|_{\ell^{v_0}(\mathcal D_1)}$, proving (c).
\end{proof}

\subsection{The region $2\le u\le \infty$ -- upper bounds}
\begin{lemma}\label{lem:rd-upper-u-ge-2}
Let $2\le u\le\infty$.
\begin{enumerate}[label=\emph{(\alph*)},itemsep=0.3em]
\item If $u\le v\le\infty$, then
\[
\|\mathcal{M}^{\mathrm{rd}}_{\mathbb H_1}F\|_{\ell^v(\mathcal D_1)}
\leq 5^{\frac{2}{u}} \,
q^{1-\frac{1}{u}}\,
\|F\|_{\ell^u(\mathbb H_1(\mathbb F_q))}.
\]
\item If $1\le v\le u$, then
\[
\|\mathcal{M}^{\mathrm{rd}}_{\mathbb H_1}F\|_{\ell^v(\mathcal D_1)}
\leq  2^{\frac{1}{v}-\frac{1}{u}} \cdot 5^{\frac{2}{u}} \, 
q^{1+\frac{2}{v}-\frac{3}{u}}\,
\|F\|_{\ell^u(\mathbb H_1(\mathbb F_q))}.
\]
\end{enumerate}
\end{lemma}

\begin{proof}
By Lemma~\ref{lem:rd-diagonal} (since $u\ge 2$),
\[
\|\mathcal{M}^{\mathrm{rd}}_{\mathbb H_1}F\|_{\ell^u(\mathcal D_1)}
\leq 5^{\frac{2}{u}} \,
q^{1-\frac{1}{u}}\,
\|F\|_{\ell^u(\mathbb H_1(\mathbb F_q))}.
\]

\noindent (a) 
If $v\ge u$, then Lemma~\ref{lem:embed} gives
$\|g\|_{\ell^v(\mathcal D_1)}\le \|g\|_{\ell^u(\mathcal D_1)}$ for all $g$ on $\mathcal D_1$.
Apply this to $g=\mathcal{M}^{\mathrm{rd}}_{\mathbb H_1}F$ to obtain (a).

\noindent (b)
If $v\le u$, then Lemma~\ref{lem:embed} gives
\[
\|g\|_{\ell^v(\mathcal D_1)}
\le
|\mathcal D_1|^{\frac{1}{v}-\frac{1}{u}}\,
\|g\|_{\ell^u(\mathcal D_1)}
\leq 2^{\frac{1}{v}-\frac{1}{u}} \,
q^{2(\frac{1}{v}-\frac{1}{u})}\,\|g\|_{\ell^u(\mathcal D_1)}.
\]
Apply this to $g=\mathcal{M}^{\mathrm{rd}}_{\mathbb H_1}F$ and combine with the diagonal bound above:
\[
\|\mathcal{M}^{\mathrm{rd}}_{\mathbb H_1}F\|_{\ell^v(\mathcal D_1)}
\leq  2^{\frac{1}{v}-\frac{1}{u}} \cdot 5^{\frac{2}{u}} \, 
q^{2(\frac{1}{v}-\frac{1}{u})}\,q^{1-\frac{1}{u}}\,
\|F\|_{\ell^u(\mathbb H_1(\mathbb F_q))}
= 2^{\frac{1}{v}-\frac{1}{u}} \cdot 5^{\frac{2}{u}} \, 
q^{1+\frac{2}{v}-\frac{3}{u}}\,
\|F\|_{\ell^u(\mathbb H_1(\mathbb F_q))}.
\]
This completes the proof.
\end{proof}

We now record four test functions forcing the four terms in \eqref{eq:Ard-formula}. 

\subsection{The case $A^{\mathrm{rd}}_1(u,v)\ \ge\ \frac{1}{v}$}
\begin{lemma}\label{lem:rd-lb-pointmass}
For every $1\le u,v\le\infty$, one has
\[
A^{\mathrm{rd}}_1(u,v)\ \ge\ \frac{1}{v}.
\]
\end{lemma}

\begin{proof}
Let $\bfp_\ast$ be any fixed point in $\mathbb H_1(\Fq)$ and take $F=\delta_{\bfp_\ast}$. Then $\|F\|_{\ell^u(\mathbb H_1(\mathbb F_q))}=1$. Note that there are exactly $q+1$ different refined directions $\omega\in \cD_1$ for the horizontal lines through $\bfp_\ast$. For such $\omega$, one has $(\mathcal M^{\mathrm{rd}}_{\mathbb H_1}F)(\omega)\geq 1$. Therefore,
\[
\|\mathcal M^{\mathrm{rd}}_{\mathbb H_1}F\|_{\ell^v(\mathcal D_1)}
\ \ge\ (q+1)^{\frac{1}{v}}\ \geq\ q^{\frac{1}{v}}.
\]
This forces $q^{A^{\mathrm{rd}}_1(u,v)}\gtrsim q^{\frac{1}{v}}$, and hence
$A^{\mathrm{rd}}_1(u,v)\ge \frac{1}{v}$.
\end{proof}

\subsection{The case $A^{\mathrm{rd}}_1(u,v)\ \ge\ 1-\frac{1}{u}$}
\begin{lemma}\label{lem:rd-lb-single-line}
For every $1\le u,v\le\infty$, one has
\[
A^{\mathrm{rd}}_1(u,v)\ \ge\ 1-\frac{1}{u}.
\]
\end{lemma}

\begin{proof}
Let $L\subset\mathbb H_1(\mathbb F_q)$ be a fixed horizontal line and set $F:=\mathbf 1_L$.
Then $\|F\|_{\ell^u(\mathbb H_1(\mathbb F_q))}=|L|^{\frac{1}{u}}=q^{\frac{1}{u}}$.
Let $\omega_0:=\mathrm{Dir}(L)$.
Since the maximum in $\mathcal M^{\mathrm{rd}}_{\mathbb H_1}$ ranges over all horizontal lines
of refined direction $\omega_0$, we have
\[
(\mathcal M^{\mathrm{rd}}_{\mathbb H_1}F)(\omega_0)\ \ge\ \sum_{\mathbf p\in L}\mathbf 1_L(\mathbf p)\ =\ q.
\]
Hence, $\|\mathcal M^{\mathrm{rd}}_{\mathbb H_1}F\|_{\ell^v(\mathcal D_1)}\ge q$ for every $v$.
Therefore $q^{A^{\mathrm{rd}}_1(u,v)}\gtrsim q/q^{\frac{1}{u}}$, and the claim follows.
\end{proof}

\subsection{The case $A^{\mathrm{rd}}_1(u,v)\ \ge\ \frac{2}{v}-\frac{1}{u}$}
\begin{lemma}\label{lem:rd-lb-blocking}
For every $1\le u,v\le\infty$, one has
\[
A^{\mathrm{rd}}_1(u,v)\ \ge\ \frac{2}{v}-\frac{1}{u}.
\]
\end{lemma}

\begin{proof}
Let $\ell_1,\ell_2\subset\mathbb F_q^2$ be two non-parallel affine lines and set
$B:=\ell_1\cup \ell_2$. Then $|B|=2q-1\sim q$.
Define
\[
E:=\{(x,y,0)\in\mathbb H_1(\mathbb F_q): (x,y)\in B\},
\qquad
F:=\mathbf 1_E.
\]
Then $\|F\|_{\ell^u(\mathbb H_1(\mathbb F_q))}=|E|^{\frac{1}{u}}\sim q^{\frac{1}{u}}$.

We claim that $(\mathcal M^{\mathrm{rd}}_{\mathbb H_1}F)(\omega)\ge 1$ for every $\omega\in\mathcal D_1$.
Fix $\omega=[a:b:c]\in\mathcal D_1$, and let $\ell_\omega\subset\mathbb F_q^2$ be the corresponding affine line.
If $\ell_\omega$ is not parallel to $\ell_1$, then $\ell_\omega\cap \ell_1\neq\emptyset$; otherwise $\ell_\omega$ intersects $\ell_2$ since $\ell_2$ is not parallel to $\ell_1$.
In either case we may choose $(x_\omega,y_\omega)\in \ell_\omega\cap B$, so $(x_\omega,y_\omega,0)\in E$.
By \eqref{eq_q_lines}, the horizontal line $L_\omega:= L_{(x_\omega,y_\omega,0),[a:b]}$ has refined direction $\omega$. 
Hence
\[
(\mathcal M^{\mathrm{rd}}_{\mathbb H_1}F)(\omega)\ \ge\ \sum_{\mathbf p\in L_\omega}\mathbf 1_E(\mathbf p)\ \ge\ 1.
\]
Therefore,
\[
\|\mathcal M^{\mathrm{rd}}_{\mathbb H_1}F\|_{\ell^v(\mathcal D_1)}
\ \ge\ |\mathcal D_1|^{\frac{1}{v}}\ \sim\ q^{\frac{2}{v}},
\]
using $|\mathcal D_1|=q^2+q$.

This forces $q^{A^{\mathrm{rd}}_1(u,v)}\gtrsim q^{\frac{2}{v}}/q^{\frac{1}{u}}$, and hence
$A^{\mathrm{rd}}_1(u,v)\ge \frac{2}{v}-\frac{1}{u}$.
\end{proof}

\subsection{The case $A^{\mathrm{rd}}_1(u,v)\ \ge\ 1+\frac{2}{v}-\frac{3}{u}$}
\begin{lemma}\label{lem:rd-lb-constant}
For every $1\le u,v\le\infty$, one has
\[
A^{\mathrm{rd}}_1(u,v)\ \ge\ 1+\frac{2}{v}-\frac{3}{u}.
\]
\end{lemma}

\begin{proof}
Let $F\equiv 1$ on $\mathbb H_1(\mathbb F_q)$. Then
$\|F\|_{\ell^u(\mathbb H_1(\mathbb F_q))}=|\mathbb H_1(\mathbb F_q)|^{\frac{1}{u}}=q^{\frac{3}{u}}$.
For each $\omega\in\mathcal D_1$, choose any horizontal line $L_\omega\in \cL(\omega)$.
Then $\sum\limits_{\mathbf p\in L_\omega}|F(\mathbf p)|=|L_\omega|=q$, so
$(\mathcal M^{\mathrm{rd}}_{\mathbb H_1}F)(\omega)=q$ for all $\omega\in\mathcal D_1$.
Hence
\[
\|\mathcal M^{\mathrm{rd}}_{\mathbb H_1}F\|_{\ell^v(\mathcal D_1)}
=
q\,|\mathcal D_1|^{\frac{1}{v}}
\sim
q\cdot q^{\frac{2}{v}}
=
q^{1+\frac{2}{v}}.
\]
Therefore $q^{A^{\mathrm{rd}}_1(u,v)}\gtrsim q^{1+\frac{2}{v}}/q^{\frac{3}{u}}$, giving the claim.

\end{proof}



\subsection{Proof of Theorem~\ref{thm:exact-rd}}

\begin{proof} [Proof of Theorem~\ref{thm:exact-rd}]
For every $1\le u,v\le\infty$, it follows from Lemmas
\ref{lem:rd-lb-pointmass}--\ref{lem:rd-lb-constant} that
\[
A^{\mathrm{rd}}_1(u,v)\ \ge\
\max\Bigg\{\frac1v,\ 1-\frac1u,\ \frac{2}{v}-\frac1u,\ 1+\frac{2}{v}-\frac{3}{u}\Bigg\}.
\]

\medskip

For the matching upper bound, Lemmas
\ref{lem:rd-upper-u-le-2} and \ref{lem:rd-upper-u-ge-2} give
\[
A^{\mathrm{rd}}_1(u,v)\ \le\
\max\Bigg\{\frac1v,\ 1-\frac1u,\ \frac{2}{v}-\frac1u,\ 1+\frac{2}{v}-\frac{3}{u}\Bigg\},
\]
and the theorem follows. 
\end{proof}

\section{Proof of Theorems \ref{thm:Nm-refined} and \ref{co-1mar}}\label{section10}

\begin{proof}[Proof of Theorem \ref{thm:Nm-refined}]
Let $F:=\mathbf 1_E$. For each $\omega\in\Omega$, the hypothesis gives a horizontal line $L$
with $\mathrm{Dir}(L)=\omega$ and $|E\cap L|\ge m$, hence
\[
(\mathcal{M}^{\mathrm{rd}}_{\mathbb H_1}F)(\omega)\ge m.
\]
Therefore,
\[
\|\mathcal{M}^{\mathrm{rd}}_{\mathbb H_1}F\|_{\ell^v(\mathcal D_1)}
\ge
\|\mathcal{M}^{\mathrm{rd}}_{\mathbb H_1}F\|_{\ell^v(\Omega)}
\ge
m\,\|\mathbf 1_\Omega\|_{\ell^v(\Omega)}
=
m\,|\Omega|^{\frac{1}{v}},
\]
with the usual convention $|\Omega|^{ \frac{1}{\infty}}=1$.

On the other hand, by \eqref{eq:Ard-def-proof},
\[
\|\mathcal{M}^{\mathrm{rd}}_{\mathbb H_1}F\|_{\ell^v(\mathcal D_1)}
\le
C_{u,v}\,q^{A^{\mathrm{rd}}_1(u,v)}\,
\|F\|_{\ell^u(\mathbb H_1(\mathbb F_q))}.
\]
Since $\|F\|_{\ell^u(\mathbb H_1(\mathbb F_q))}=|E|^{\frac{1}{u}}$, we obtain
\[
m\,|\Omega|^{\frac{1}{v}}
\le
C_{u,v}\,q^{A^{\mathrm{rd}}_1(u,v)}\,|E|^{\frac{1}{u}}.
\]
Rearranging gives \eqref{eq:Nm-uv-lower}. This completes the proof.
\end{proof}

\begin{proof}[Proof of Theorem \ref{co-1mar}]
By Theorem~\ref{thm:exact-rd},
\[
A^{\mathrm{rd}}_1\!\left(\frac{s}{s-1},s\right)=\frac1s.
\]
Applying \eqref{eq:Ard-def-proof} with
$u=\frac{s}{s-1}$ and $v=s$, we have
\[
\|M_E\|_{\ell^s(\mathcal D_1)}
\le
C_{\frac{s}{s-1},\,s}\,q^{\frac{1}{s}}\,|E|^{1-\frac{1}{s}},
\]
and raising to the $s$th power gives the desired result.
\end{proof}

\section{Discussions and examples}\label{section11}
\subsection{Examples distinguishing affine Kakeya and refined-direction horizontal Kakeya}
\label{sec:examples-affine-vs-refined}

In this subsection, we record examples showing that the usual affine Kakeya in $\mathbb{F}_q^3$ and the full
refined-direction horizontal Heisenberg Kakeya in $\mathbb{H}_1(\mathbb{F}_q)$ are
incomparable: neither condition implies the other.

We use the standard coordinate identification of sets
\[
\mathbb{H}_1(\mathbb{F}_q)\cong \mathbb{F}_q^3,\qquad \mathbf{p}=(x,y,t).
\]
Recall that the set of refined directions is
\[
\mathcal{D}_1=\{[a:b:c]\in\mathbb{P}^2(\mathbb{F}_q): (a,b)\neq (0,0)\}.
\]
{ 
If \(L\) is a horizontal line, \((a,b)\neq(0,0)\) is a chosen representative of its horizontal direction, and \(\mathbf p_0=(x_0,y_0,t_0)\in L\), then
\[
\mathrm{Dir}(L)=[a:b:x_0b-y_0a]\in\mathcal{D}_1.
\]
}
As discussed at the beginning of Section \ref{section8}, every $\omega\in\mathcal{D}_1$ admits a unique
normal form
\[
\omega=[1:m:\gamma]\quad (m,\gamma\in\mathbb{F}_q),
\qquad\text{or}\qquad
\omega=[0:1:\gamma]\quad (\gamma\in\mathbb{F}_q),
\]
and the family of horizontal lines with refined direction $\omega$ is precisely
\begin{align*}
L_{\omega,\tau}
&=\{(x,\ \textcolor{black}{mx-\gamma},\ \tau+\gamma x): x\in\mathbb{F}_q\},
\qquad &&\text{if }\omega=[1:m:\gamma],\\
L_{\omega,\tau}
&=\{(\textcolor{black}{\gamma},\ y,\ \tau+\gamma y): y\in\mathbb{F}_q\},
\qquad &&\text{if }\omega=[0:1:\gamma],
\end{align*}
where $\tau\in\mathbb{F}_q$ parametrizes the $q$ distinct horizontal lines of refined direction $\omega$.

The first example shows an affine Kakeya set that is not full refined-direction horizontal Kakeya.  

\begin{example}
\label{ex:affine-not-refined}
Fix $\omega_0\in\mathcal{D}_1$. \textcolor{black}{Write $\omega_0$ in normal form as $\omega_0=[1:m_0:\gamma_0]$ or $\omega_0=[0:1:\gamma_0]$.}
Define a subset $S_{\omega_0}\subset\mathbb{H}_1(\mathbb{F}_q)$ by
\[
S_{\omega_0}:=
\begin{cases}
\{(0,\ \textcolor{black}{-\gamma_0},t): t\in\mathbb{F}_q\}, & \text{if }\omega_0=\textcolor{black}{[1:m_0:\gamma_0]},\\[0.2em]
\{(\textcolor{black}{\gamma_0},0,t): t\in\mathbb{F}_q\}, & \text{if }\omega_0=\textcolor{black}{[0:1:\gamma_0]}.
\end{cases}
\]
Let $E:=\mathbb{H}_1(\mathbb{F}_q)\setminus S_{\omega_0}$. Then, viewed as a subset of
$\mathbb{F}_q^3$, the set $E$ is an affine Kakeya set, but $E$ is not a full refined-direction
horizontal Heisenberg Kakeya set in the sense of Definition~\ref{def:slope-refined-kakeya}.
\end{example}

\begin{proof}
We first verify that $E$ is not full refined-direction horizontal Heisenberg Kakeya.  Consider the
family $\{L_{\omega_0,\tau}\}_{\tau\in\mathbb{F}_q}$ of all horizontal lines with refined direction
$\omega_0$.
If $\omega_0=\textcolor{black}{[1:m_0:\gamma_0]}$, then for every $\tau\in\mathbb{F}_q$ one has
\[
(0,\ \textcolor{black}{-\gamma_0},\tau)\in L_{\omega_0,\tau}\cap S_{\omega_0}.
\]
If $\omega_0=\textcolor{black}{[0:1:\gamma_0]}$, then for every $\tau\in\mathbb{F}_q$ one has
\[
(\textcolor{black}{\gamma_0},0,\tau)\in L_{\omega_0,\tau}\cap S_{\omega_0}.
\]
In either case, every horizontal line $L$ with $\mathrm{Dir}(L)=\omega_0$ meets $S_{\omega_0}$, and
hence no such line can be contained in $E$.  Therefore, $E$ fails the defining condition in
Definition~\ref{def:slope-refined-kakeya} at the refined direction $\omega_0$.

We next show that $E$ is affine Kakeya in $\mathbb{F}_q^3$.  Fix an ambient direction
$[\mathbf{v}]\in\mathbb{P}^2(\mathbb{F}_q)$ and choose a representative $\mathbf{v}\in\mathbb{F}_q^3\setminus\{0\}$.
The affine lines in direction $[\mathbf{v}]$ are the cosets of the 1-dimensional subspace
$\langle \mathbf{v}\rangle$, and these $q^2$ lines form a partition of $\mathbb{F}_q^3$.
Since $|S_{\omega_0}|=q$, at most $q$ of these $q^2$ lines meet $S_{\omega_0}$ (each point lies on
exactly one line of direction $[\mathbf{v}]$).  Hence, there exists an affine line $\ell$ of ambient
direction $[\mathbf{v}]$ such that $\ell\cap S_{\omega_0}=\varnothing$, and therefore $\ell\subset E$.
As $[\mathbf{v}]$ was arbitrary, $E$ contains an affine line in every ambient direction, i.e.\ $E$ is
an affine Kakeya set.
\end{proof}

The second example gives a full refined-direction horizontal Kakeya set that is not affine Kakeya.

\begin{example}
\label{ex:refined-not-affine}
\textcolor{black}{Assume $q$ is odd and $q>3$.}  For each $\omega\in\mathcal{D}_1$ in normal form, define a horizontal line $L_\omega$ by
\[
L_{[1:m:\gamma]}:=\{(x,\ \textcolor{black}{mx-\gamma},\ m^2+\gamma x): x\in\mathbb{F}_q\},
\qquad
L_{[0:1:\gamma]}:=\{(\textcolor{black}{\gamma},\ y,\ \gamma y): y\in\mathbb{F}_q\}.
\]
Let
\[
E:=\bigcup_{\omega\in\mathcal{D}_1} L_\omega\ \subset\ \mathbb{H}_1(\mathbb{F}_q).
\]
Then, $E$ is a full refined-direction horizontal Heisenberg Kakeya set (Definition~\ref{def:slope-refined-kakeya}),
but $E$ is not an affine Kakeya set in $\mathbb{F}_q^3$.
\end{example}

\begin{proof}
It is clear that $E$ is full refined-direction horizontal Heisenberg Kakeya.

We now show that $E$ is not affine Kakeya by proving that $E$ contains no vertical affine line, i.e.\
no affine line of direction $[0:0:1]\in\mathbb{P}^2(\mathbb{F}_q)$.  Fix $(x_0,y_0)\in\mathbb{F}_q^2$
and consider the vertical fiber
\[
V_{(x_0,y_0)}:=\{(x_0,y_0,t): t\in\mathbb{F}_q\}.
\]
It suffices to show that $V_{(x_0,y_0)}\nsubseteq E$ for every $(x_0,y_0)$.

Define
\[
\mathcal T(x_0,y_0):=\{t\in\mathbb{F}_q:\ (x_0,y_0,t)\in E\}.
\]
We bound $|\mathcal T(x_0,y_0)|$.  First consider the contribution from directions of the form $\omega=[1:m:\gamma]$.
For a fixed $m\in\mathbb{F}_q$, the point $(x_0,y_0,t)$ can lie on $L_{[1:m:\gamma]}$ only if
$y_0=\textcolor{black}{mx_0-\gamma}$, i.e.\ $\textcolor{black}{\gamma=mx_0-y_0}$.  For this value of $\gamma$, the point of
$L_{[1:m:\gamma]}$ with $x=x_0$ has
\[
t=m^2+\gamma x_0
=
\textcolor{black}{m^2+(mx_0-y_0)x_0}
=
\textcolor{black}{m^2+x_0^2 m-x_0y_0}.
\]
Hence, the set of $t$-values arising from all such $\omega$ is contained in
\[
A(x_0,y_0):=
\left\{\textcolor{black}{m^2+x_0^2 m-x_0y_0}:\ m\in\mathbb{F}_q\right\}.
\]
Since $q$ is odd, completing the square shows that $A(x_0,y_0)$ is a translate of the set of squares
in $\mathbb{F}_q$, and therefore
\[
|A(x_0,y_0)|\le \frac{q+1}{2}.
\]

Next consider directions of the form $\omega=[0:1:\gamma]$.  For any $x_0\in\mathbb{F}_q$ there is a
unique $\textcolor{black}{\gamma=x_0}$ such that $L_{[0:1:\gamma]}$ has first coordinate equal to $x_0$.  Intersecting
\[
L_{[0:1:\textcolor{black}{x_0}]}=\big\{ \big(x_0,\ y,\ \textcolor{black}{x_0 y} \big): y\in\mathbb{F}_q \big\}
\]
with $y=y_0$ yields at most one additional value $\textcolor{black}{t= x_0 y_0}$.

Consequently,
\[
|\mathcal T(x_0,y_0)|
\le
|A(x_0,y_0)|+1
\le
\frac{q+1}{2}+1
=
\frac{q+3}{2}
<
q
\qquad\text{since }q>3.
\]
Thus, $V_{(x_0,y_0)}\nsubseteq E$ for all $(x_0,y_0)$, so $E$ contains no vertical affine line.
In particular, $E$ does not contain an affine line in the ambient direction $[0:0:1]$, and hence $E$
is not an affine Kakeya set in $\mathbb{F}_q^3$.
\end{proof}

The preceding examples show that, although $\mathbb{H}_1(\mathbb{F}_q)$ and $\mathbb{F}_q^3$ coincide
as sets, the affine Kakeya property in $\mathbb{F}_q^3$ and the full refined-direction horizontal
Heisenberg Kakeya property in $\mathbb{H}_1(\mathbb{F}_q)$ are incomparable. The next example shows that the $\ell^3\to\ell^3$ bound is sharp.

\begin{example} \label{ex:rd-l3-sharp}
Let $F\equiv 1$ on $\mathbb{H}_1(\mathbb{F}_q)$, i.e.\ $F=\mathbf{1}_{\mathbb{H}_1(\mathbb{F}_q)}$.
Then
\[
\|\mathcal{M}^{\mathrm{rd}}_{\mathbb{H}_1}F\|_{\ell^3(\mathcal{D}_1)}
=
q\,|\mathcal{D}_1|^{\frac{1}{3}}
\sim
q^{\frac{5}{3}},
\qquad
\|F\|_{\ell^3(\mathbb{H}_1(\mathbb{F}_q))}
=
|\mathbb{H}_1(\mathbb{F}_q)|^{\frac{1}{3}}
=
q.
\]
Consequently,
\[
\frac{\|\mathcal{M}^{\mathrm{rd}}_{\mathbb{H}_1}F\|_{\ell^3(\mathcal{D}_1)}}
{\|F\|_{\ell^3(\mathbb{H}_1(\mathbb{F}_q))}}
=
|\mathcal{D}_1|^{\frac{1}{3}}
=
(q^2+q)^{\frac{1}{3}}
\geq
q^{\frac{2}{3}}.
\]
In particular, any estimate of the form
\[
\|\mathcal{M}^{\mathrm{rd}}_{\mathbb{H}_1}F\|_{\ell^3(\mathcal{D}_1)}
\le C\,q^\alpha\,\|F\|_{\ell^3(\mathbb{H}_1(\mathbb{F}_q))}
\]
(valid uniformly in $q$) forces $\alpha\ge \frac23$. Hence, the exponent $q^{\frac{2}{3}}$ in the
$\ell^3\to\ell^3$ bound is sharp up to absolute constants.
\end{example}

\begin{proof}
Fix $\omega\in\mathcal{D}_1$.  As above, there exists a horizontal line $L$ with $\mathrm{Dir}(L)=\omega$,
and $|L|=q$.  Since $F\equiv 1$, it follows that
\[
(\mathcal{M}^{\mathrm{rd}}_{\mathbb{H}_1}F)(\omega)
=
\max_{\substack{L\ \mathrm{horizontal}\\ \mathrm{Dir}(L)=\omega}}
\sum_{\mathbf{p}\in L}|F(\mathbf{p})|
=
q.
\]
Thus, $\mathcal{M}^{\mathrm{rd}}_{\mathbb{H}_1}F\equiv q$ on $\mathcal{D}_1$, and therefore
\[
\|\mathcal{M}^{\mathrm{rd}}_{\mathbb{H}_1}F\|_{\ell^3(\mathcal{D}_1)}
=
\Bigl(\sum_{\omega\in\mathcal{D}_1} q^3\Bigr)^{\frac{1}{3}}
=
q\,|\mathcal{D}_1|^{\frac{1}{3}}.
\]
Also $|\mathbb{H}_1(\mathbb{F}_q)|=q^3$, so $\|F\|_{\ell^3(\mathbb{H}_1(\mathbb{F}_q))}=q$.

Finally, $|\mathcal{D}_1|=q^2+q$, hence $|\mathcal{D}_1|^{\frac{1}{3}}=(q^2+q)^{\frac{1}{3}}\geq q^{\frac{2}{3}}$, which implies the stated sharpness.
\end{proof}

\subsection{A heuristic connection with the affine Kakeya problem in $\mathbb{F}_q^3$}\label{sec11.2}
Let \(E\subset \mathbb F_q^3\) be an affine Kakeya set, and identify \(\mathbb F_q^3\) with the underlying set of \(\mathbb H_1(\mathbb F_q)\).
Write \(\mathcal D_1=\mathbb P^2(\mathbb F_q)\setminus\{[0:0:1]\}\) for the set of non-vertical directions, and decompose
\(\mathcal D_1=\Omega_1\sqcup\Omega_2\) as follows: \(\Omega_1\) consists of those \(\omega\in\mathcal D_1\) for which \(E\) contains a horizontal line
\(L\subset \mathbb H_1(\mathbb F_q)\) with \(\mathrm{Dir}(L)=\omega\), and \(\Omega_2:=\mathcal D_1\setminus\Omega_1\).

Theorem~\ref{thm:Nm-refined} immediately controls the \(\Omega_1\) contribution. More precisely,
\[
|E|\ \gtrsim\ q\,|\Omega_1|.
\]
In particular, \(|\Omega_1|\gtrsim q^2\) already forces \(|E|\gtrsim q^3\).
Since \(|\mathcal D_1|=q^2+q\), at least one of \(\Omega_1\) or \(\Omega_2\) has size \(\gtrsim q^2\). Furthermore, note that the second chart $[0:1:\gamma]$ contains at most $q$ directions. Denote by $\Omega_2'$ the set of directions in $\Omega_2$ with first chart $[1:m:\gamma]$. For the purpose of proving a lower bound of
order \(q^3\), it suffices to treat the complementary range
\begin{equation}\label{eq:Omega2-large}
|\Omega_2'|\ \gtrsim\ q^2.
\end{equation}
In this sense, the affine Kakeya problem in \(\mathbb F_q^3\) is reduced to understanding an \(\Omega_2'\)-affine Kakeya configuration, in which a large
proportion of directions is realized only by non-horizontal lines in \(\mathbb H_1(\mathbb F_q)\).

The advantage of this viewpoint is that the $\Omega_2'$ lines satisfy a concrete structural constraint.
Fix $\omega=[a:b:c]\in\Omega_2$ with $(a,b)\neq(0,0)$, and let
\[
\ell_\omega=\{(x_0,y_0,t_0)+s(a,b,c):s\in\F_q\}\subset E
\]
be a Kakeya line of ambient direction $\omega$ contained in $E$.
By definition of $\Omega_2'$, the line $\ell_\omega$ is not horizontal in $\mathbb H_1(\F_q)$.
Equivalently, its basepoint fails the horizontality relation
\[
bx_0-ay_0\ \neq\ c.
\]
Thus, the Kakeya line in direction $[a:b:c]$ is forced to choose its basepoint away from the affine line
$\{(x,y)\in\F_q^2:\ bx-ay=c\}$ in the $(x,y)$ plane. We refer to this as a basepoint-avoidance constraint.

To quantify this failure of horizontality in a way that is intrinsic to the affine line (independent of the chosen point on the line and independent of scaling of $(a,b,c)$), we define a scalar parameter $\mu(\ell)\in\F_q$ using the normal-form coordinates on $\mathcal D_1$.
Write the direction of $\ell$ in the unique normal form $\omega=[1:m:\gamma]$ for $m,\gamma\in\F_q$. Pick any point $(x_0,y_0,t_0)\in\ell$, and set
\[
\mu(\ell)\;:=\;\gamma-(mx_0-y_0)\in\F_q.
\]
A direct check shows that $\mu(\ell)$ does not depend on the choice of $(x_0,y_0,t_0)\in\ell$.
Moreover, $\mu(\ell)=0$ if and only if $\ell$ is horizontal in $\mathbb H_1(\F_q)$. In particular, for $\omega\in\Omega_2$ with the first chart, we necessarily have $\mu(\ell_\omega)\in\F_q^*$.

Now fix a choice of the Kakeya line $\ell_\omega\subset E$ for each $\omega\in\Omega_2'$, and define
\[
k(\omega):=\mu(\ell_\omega)\in\F_q^*,
\qquad
\Lambda_k:=\{\omega\in\Omega_2':\ k(\omega)=k\}.
\]
Then, $\Omega_2'=\bigsqcup\limits_{k\in\F_q^*}\Lambda_k$ is a partition.

For a fixed $k\in\F_q^*$, lines with $\mu(\ell)=k$ can be straightened into horizontal lines by an explicit mapping:
 $(x,y,t)\mapsto(x,y,t-kx)$.
Thus, each slice $\Lambda_k$ may be viewed (after such a mapping) as contributing a horizontal Kakeya configuration, to which Theorem~\ref{thm:sr-l2-bound} applies.
In particular, if $|\Lambda_k|\gtrsim q^2$ for some $k$, then this immediately forces $|E|\gtrsim q^3$.

In the remaining range, where $|\Omega_2'|\gtrsim q^2$, but each individual slice $|\Lambda_k|$ is much smaller than $q^2$, one must exploit how the different $\mu$-slices overlap inside $E$.
Equivalently, one needs a mechanism that controls cross-incidences between the families of lines with $\mu=k$ and $\mu=k'$ for $k\neq k'$.
This is the point at which the basepoint-avoidance constraint may provide additional structure not visible in the usual affine formulation.

{\bf Acknowledgement.} 
D. T. Tran would like to thank the Vietnam Institute for Advanced Study in Mathematics (VIASM) for its warm hospitality and excellent research environment, where part of this work was completed during his stay. A.P. is a member of the Istituto Nazionale di Alta Matematica (INdAM), Gruppo Nazionale per l'Analisi Matematica, la Probabilità e le loro Applicazioni (GNAMPA), and is supported by the University of Trento and the INdAM-GNAMPA 2025 Project \emph{Structure of sub-Riemannian hypersurfaces in Heisenberg groups}, CUP ES324001950001.

\end{document}